\documentclass{article}[12pt]
\usepackage{amsfonts}
\usepackage[a4paper]{geometry}
\usepackage{amsmath}
\usepackage{amssymb}
\usepackage{graphicx}
\usepackage{color}
\usepackage{colortbl}
\def\dd{{\rm d}}

\newcommand{\R}{\mathbb{R}}

\newcommand{\PP}{\mathbb{P}}
\newcommand{\F}{\mathcal{F}}
\newcommand{\V}{\mathcal{V}}
\newcommand{\U}{\mathcal{U}}
\newcommand{\I}{\mathcal{I}}
\newcommand{\FF}{\mathbb{F}}
\newcommand{\EE}{\mathbb{E}}

\newtheorem{theorem}{Theorem}

\newtheorem{corollary}{Corollary}
\newtheorem{remark}{Remark}
\newtheorem{lemma}{Lemma}
\newtheorem{example}{Example}

\numberwithin{equation}{section}
\numberwithin{theorem}{section}
\numberwithin{proposition}{section}
\numberwithin{corollary}{section}
\numberwithin{remark}{section}
\numberwithin{lemma}{section}
\numberwithin{definition}{section}
\newenvironment{proof}[1][Proof]{\noindent \textbf{#1.} }{\ \rule{0.5em}{0.5em}}

\begin{document}
\parindent=0pt

\title{Metric regularity  under G\^ateaux differentiability with applications to optimization and stochastic optimal control problems }

\author{
A. Jourani \thanks{Institut de Math\'ematiques de Bourgogne, UMR-CNRS 5584, Universit\'e de Bourgogne, 21078 Dijon, France (Abderrahim.Jourani@u-bourgogne.fr).}  \and F. J.
Silva \thanks{ Institut de recherche XLIM-DMI, UMR-CNRS 7252, Universit\'e de Limoges, 87060 Limoges, France
(francisco.silva@unilim.fr).}}

\maketitle

\begin{abstract}
The main objective of this work is to study the existence of Lagrange multipliers for infinite dimensional problems under G\^ateux differentiability assumptions on the data. Our investigation follows two main steps: the proof of the existence of Lagrange multipliers under a calmness assumption on the constraints and the study of sufficient conditions, which only use the G\^ateaux derivative of the function defining the constraint, that ensure this assumption. 

We apply the abstract results to recover in a direct manner the optimality systems associated to two types of standard stochastic optimal control problems. 
\end{abstract}\smallskip

\section{Introduction}

 Consider the following
optimization problem
\begin{eqnarray}\label{PD}
\min \{f(x) \; ; \; \;   g(x) \in D\},
\end{eqnarray}
where $f : X \to  \R $ and $g : X \to Y$ are (for simplicity of the exposition) differentiable mappings, $X$ and $Y$ are Banach spaces and $D\subseteq Y$  is nonempty. In the case where $Y$ is finite dimensional the following result holds for any
closed set $D$ : if $x_0$ is a local solution  to $(P)$, then there are $\lambda \geq 0$ and
$y^* \in N(D, g(x_0))$ such that
\begin{equation}\label{optimality_condition_expected_non_zero_multipliers}(\lambda, y^*) \not = (0, 0), \end{equation}
\begin{equation}\label{optimality_condition_expected}\lambda f'(x_0) + y^* \circ g'(x_0) = 0.\end{equation}
Here $N(D, g(x_0)) $ denotes some normal cone to $D$ at $g(x_0)$ (say, for instance, the  Clarke
normal cone, the  approximate normal cone, etc..). 

The following example proposed by Brokate in \cite[Section 2]{B} shows that the previous  result is no longer true in the infinite dimensional
case. 

\begin{example}\label{Brok}   Let  $X = Y =  \ell^2$ be the Hilbert space of square summable
real sequences.  Denote by $(e_k)_{k\geq 1}$ the canonical orthonormal base of $\ell^2$ and consider the  operator $A :  \ell^2  \to  \ell^2 $  defined by
$$A\left(\sum_{ i\geq 1 } x_i e_i\right) = \sum_{ i\geq 1 } 2^{1-i}x_ie_i.$$
It is easy to check that $A$  is injective but not surjective and  that  the image of $A$, denoted by  $\mbox{{\rm Im}}(A)$, is a proper dense subspace of $ \ell^2$.  As a consequence, the adjoint operator $A^*$ is injectif but not surjectif. Now,  let $x^* \in  \ell^{2} \setminus  {\rm Im}(A^*)$ and  consider  $f = x^*$, $g = A$ and $D =
\{0\}$  as the data for problem \eqref{PD}. Then $0$ is  the  only feasible point and, hence,  the solution of this problem. Moreover, since $x^* \in  \ell^{2} \setminus  {\rm Im}(A^*)$, we easily check that  there is no $(\lambda, y^*) \neq (0, 0)$ satisfying  $(1.4)$.
\end{example}

In infinite dimension, most of the authors have assumed that $D$ is a closed convex cone with
a nonempty interior or  that $D = D_1 \times \{0\}$, where $D_1$ is a closed convex cone with  a 
nonempty interior and $\{0\} \subset  \R ^n$ (see \cite{ELTH,Gl,Ish,Mi,Yu} and references therein). The first result which gives a condition  for the validity of \eqref{optimality_condition_expected_non_zero_multipliers}-\eqref{optimality_condition_expected}  in the
case  where $D$ is closed is  due to Jourani and Thibault \cite{JT1}, where it  is assumed that the system $g(x)
\in D$ is {\it metrically regular}  (see \cite{Dontchev_Rockafellar_metric_regularity,Ioffe_book}   and the references therein for a systematic study of this property). This condition is expressed metrically in terms of $g$ and $D$
and implies that  $\lambda$ can be taken different from  zero. In \cite{J1} it is shown that relations \eqref{optimality_condition_expected_non_zero_multipliers}-\eqref{optimality_condition_expected} 
subsist in the case where $f$ is vector-valued  and $D$ is epi-Lipschitz-like in the sense of
Borwein (see \cite{B1}). In \cite{JT2, JT3}, the authors  gave general conditions ensuring \eqref{optimality_condition_expected_non_zero_multipliers} and \eqref{optimality_condition_expected}.  More
precisely, let $x_0$ be a local solution to problem $(P)$ and suppose that $f$ and $g$ are locally Lipschitz  mappings at $x_0$, with $g$ strongly compactly
Lipschitz at $x_0$ (see \cite{JT1}). Denote by $\partial_A d(u, D)$ the approximate subdifferential of $d(\cdot,D)$ at $u$ (see \cite{I1, I2}) and assume the existence of a
locally compact cone $K^* \subset Y^*$ and a neighbourhood $V$ of $g(x_0)$ such that
$$\partial_Ad(u, D) \subset K^*, \quad \forall \;  u \in V\cap D,$$
or equivalently (see \cite{I3}), $D$ is compactly epi-Lipschitzian in the sense of Borwein-Strojwas \cite{BS}. Then there exists $\lambda \geq 0$ and $y^* \in   \R_ +\partial_Ad(g(x_0), D)$, with
$(\lambda, y^*) \neq (0, 0)$, such that
$$\lambda \partial_A f(x_0) + \partial_A(y^* \circ g)(x_0) \ni 0.$$

Now, in order to ensure the existence of Lagrange multipliers (i.e. $\lambda \neq 0$ in \eqref{optimality_condition_expected}), several qualification conditions have been considered in the literature, including the classical ones as Slater condition, Mangasarian-Fromovitz  condition and so on. In this paper, we are interested in the existence of  Lagrange  multipliers for problem (\ref{PD}), where  the problem  is nonconvex,   the data is G\^ateaux differentiable and the set $D$ is a closed set.  These multipliers are obtained in Theorem \ref{thm2} and in Theorem \ref{thm2_b} under the so-called calmness condition which is a kind of constraint qualification.  Inspired by the work by Ekeland \cite{Ekeland_inverse}, our main results (Theorem \ref{thm3} and Theorem \ref{casDconvexe}) establish the metric regularity property for the constraint system  under G\^ateaux differentiability assumptions only. We point out that the proofs of these results do not rely on any iteration scheme. 

The first application of our results are first order necessary optimality conditions for  stochastic optimal control problems in continuous time. Following the functional framework proposed by Backhoff and Silva in \cite{backhoffsilva14},  our abstract results allow us to recover a weak version of the general stochastic Pontryagin's maximum principle, proved in \cite{Peng90}, under rather general assumptions (see our Remark \ref{comments_result} {\rm(ii)}). As pointed out in \cite{backhoffsilva14}, the main difficulty in deriving this result, from standard variational principles, is that the smoothness of the equality constraint that defines the dynamics of the controlled diffusion process is difficult to check. Our abstract results, which assume only G\^ateaux differentiability of the mapping that defines the constraints and a uniform surjectivity property of the G\^ateaux derivative in a neighbourhood of the optimal solution, allow us to avoid this issue and to establish the existence of Lagrange multipliers and, as a consequence of the characterization of these multipliers studied  in \cite{backhoffsilva14}, the validity of a weak version of Pontryagin's principle. We point that this result is not new and is weaker than the one proved in \cite{Peng90}, which, however, needs strong assumptions on the second order derivatives of the data. On the other hand, the proof presented here is new,  short and clarifies the role of the adjoint states as Lagrange multipliers when the stochastic control problem is formulated in the correct functional framework.

    In the second application, we consider a discrete time stochastic optimal control problem where the randomness is modelled by a multiplicative independent noise. As in the continuous time case,  the main difficulty to apply standard abstract Lagrange multiplier  results comes from the functional equation defining the controlled trajectory. By considering a suitable functional framework for the optimization problem and using our abstract results, we are able to prove in a rather straightforward manner the validity of the optimality system obtained in \cite{LinZhang} under more general assumptions than those imposed in that article.

The paper is organized as follows. In the next section  we set up the notation and recall some standard results in nonsmooth analysis. In Section \ref{optimality_conditions_abstract_form}, we establish the existence of Lagrange multipliers for problem \eqref{PD} under the calmness assumption. Next, in Section \ref{metric_regularity}, we provide sufficient conditions, in terms of the G\^ateaux derivative of $g$, for the metric regularity of the constraint system (which is a stronger property than its calmness). Finally, in Sections \ref{section_stochastic_control_continuous_time} and \ref{section_discrete_time}, we apply these abstract results to the stochastic control problems described in the previous paragraphs.

\section{Notations and preliminaries}

In all the paper $(X, \|\cdot\|_{X})$  and  $(Y, \|\cdot\|_{Y})$  are (real) Banach spaces. The dual spaces of  $X$ and $Y$ are denoted by $X^{\ast}$ and $Y^{\ast}$, respectively, and for $h\in X$ we  set $\langle x^{\ast},h\rangle_{X}:= x^{\ast}(h)$, with an analogous notation for the duality paring between $Y^{\ast}$ and $Y$.  Given $r>0$  and $x\in X$ we denote $B_{X}(x,r):= \{ x'\in X \; ; \; \|x'-x\|_{X} \leq r\}$ the closed ball of radius $r$ centered at $x$. For $A\subseteq X$ we denote by $\mbox{cl}(A)$ and $\mbox{int}(A)$ its closure and its topological interior, respectively. \\
Let us recall some basic notions in nonsmooth analysis  (see e.g. \cite{Clarke90,Bonnans_Shapiro_book,MR2191744} for a detailed account of the theory). Given a locally Lipschitz function $\varphi: X \mapsto \R$,  the directional derivative $\varphi^{\circ}(x; h)$ of $\varphi$ at $x$ in the direction $h\in X$ and the subdifferential $\partial_{C} \varphi(x)$ of $\varphi$ at $x$  are both defined in sense of {\it Clarke}  as 
$$\begin{array}{l}\displaystyle \varphi^{\circ}(x; h):= \limsup_{y\to x,   \tau \downarrow 0} \frac{\varphi(y+\tau h)- \varphi(y)}{\tau}, \\[12pt]
\partial_{C} \varphi(x):= \left\{ x^{\ast} \in X^{\ast} \; ; \;\langle x^{\ast}, h \rangle_{X}  \leq \varphi^{\circ}(x; h) \; \; \forall \; h \in X\right\}. \end{array}$$
Note that for all $x\in X$,  $\varphi^{\circ}(x; \cdot): X \mapsto \R$ is well-defined,  positively homogeneous, subadditive, Lipschitz continuous  and  satisfies that  $\varphi^{\circ}(x; 0)=0$. This implies that  $\varphi^{\circ}(x; \cdot)$ is the support function of $\partial_{C} \varphi(x)$, which is a   nonempty, weak$^\ast$-compact  and convex set  (see \cite[Proposition 2.1.2]{Clarke90}). Given  a nonempty  set   $A\subseteq X$, we denote by $d_{A}(\cdot):= \inf_{x \in A} \|(\cdot) -x\|$ the distance to $A$  function.   Given $x\in  \mbox{cl}(A)$,  the {\it Clarke's}  tangent cone is defined as 
$$
T_{A}(x):= \left\{ h\in X : \quad   \lim_{y \to x, \; y \in A,\;  \tau \to 0^+}  {d_A(y+\tau h) \over \tau} = 0\right\}.
$$
%
If $x\notin \mbox{cl}(A)$ we set $T_{A}(x):=\emptyset$. If $x \in  \mbox{cl}(A)$, we have that $h \in T_{A}(x)$ iff for every sequences $(x_n)$ such that $x_n\in A$, $x_n\to x$, and $\tau_n \to 0^+$ there exists a sequence $h_n \to h$ such that $x_n +\tau_n h_n \in A$ for  all $n$ large enough.  The {\it Clarke's} normal cone  to $A$ at $x$ is defined as $N_{A}(x)= T_{A}(x)^0$, where for a given cone $\mathcal{K}$ we denote by $\mathcal{K}^0$ its negative polar cone, defined as
$$\mathcal{K}^0 : = \{ x^*\in X^*: \quad \langle x^*, h\rangle_{X} \leq 0 \, \hspace{0.2cm}  \forall \;  h\in \mathcal{K}\}.$$
We have  (see e.g.  \cite[Proposition 2.4.2]{Clarke90})
\begin{equation}\label{normalconeanddistance}
\begin{array}{l}
N_{A}(x)=  \mbox{$w^*$-cl}\left( \bigcup_{\lambda \geq 0} \lambda \partial_{C} d_{A}(x)\right),
\end{array}
\end{equation}
where $\mbox{$w^*$-cl}$ denotes the weak-star closure in  $X^{\ast}$.  The  {\it adjacent} (or {\it Ursescu}) tangent cone to $A$ at $x\in \mbox{cl}(A)$ is defined by 
$${\cal T}(A, x) = \left\{ h\in X : \quad   \lim_{\tau \to 0^+}{{d_A(x+\tau h)}\over \tau} = 0\right\}.$$
We set  ${\cal T}(A,x):=\emptyset$ if $x\notin  \mbox{cl}(A)$. By definition,  if $x\in \mbox{cl}(A)$  then $h\in {\cal T}(A,x)$ iff  for any sequence $\tau_n\to 0^+$ there exists a sequence $h_n\to h$ such that $x+\tau_nh_n \in A \hbox{  for all $n$ sufficiently large}$.


Finally, the {\it contingent}  (or {\it Bouligand})  tangent cone to $A$  at $x\in \mbox{cl}(A)$  is defined  as 
 $$K(A, x) := \left\{ h\in X : \quad d^-_A(x  ;  h)  = 0\right\},$$
where $d^-_A(x  ;  h)$ is the lower Dini directional derivative of $d_A$ at $x$ in the direction $h$, that is,
$$d^-_A(x  ;  h) :=\liminf_{\tau \to 0^+} {{d_A(x+\tau h)}\over \tau}.$$
We set  $K(A,x):=\emptyset$ if $x\notin  \mbox{cl}(A)$. By definition, if $x\in \mbox{cl}(A)$  then $h\in K(A, x)$ iff there exist sequences $\tau_n\to 0^+$ and $h_n\to h$ such that  $x+\tau_nh_n \in A \hbox{  for $n$ sufficiently large}$.
 Note that  
$$
T_{A}(x) \subseteq {\cal T}(A, x) \subseteq K(A,x).
$$
If $A$ is convex, then the previous tangent cones coincide. In the general  case these cones are closed, they differ and only $T_{A}(x)$ is guaranteed to be convex.

We say that $A$ is {\it tangentially regular} at $x$ if 
\begin{eqnarray}\label{regT}
K(A, x) = {\cal T}(A, x).
\end{eqnarray}
For later use, we state the following result whose proof can be easily deduced from the previous definitions. 
\begin{lemma}\label{lemmaP} Let $A\subset X$ and $B\subset Y$ be closed sets and let $x_0\in A$ and $y_0\in B$. The space $X\times Y$ is endowed with the product norm, that is, $\Vert (x, y)\Vert_{X\times Y} = \Vert x\Vert_X + \Vert y\Vert_Y$. Then 
\begin{itemize}
\item[{\rm(i)}] $K(A\times B, (x_0, y_0)) \subset K(A, x_0)\times K(B, y_0)$. The equality holds whenever  $A$  is tangentially regular at $x_0$ or   $B$ is tangentially regular at $y_0$. 
\item[{\rm(ii)}] For all $h\in X$ and $k\in Y$, $d^-_{A\times B}((x_0, y_0), (h, k)) \leq d^-_A(x_0, h) + d^0_B(y_0, k)$.
\item[{\rm(iii)}] If $A$  is tangentially regular at $x_0$ or   $B$ is tangentially regular at $y_0$, then for all $h\in X$ and $k\in Y$, 
$$d^-_{A\times B}((x_0, y_0), (h, k)) \leq d_{K(A, x_0)}( h) + d_{K(B, y_0)}( k).$$
\end{itemize}
\end{lemma}

\section{Lagrange multipliers for optimization problems  under G\^ateaux differentiability assumptions on the data}\label{optimality_conditions_abstract_form}

This section is concerned with necessary optimality conditions or existence of  Lagrange  multipliers  associated to local solutions of optimization problems of the form
\begin{eqnarray}\label{opt}
\left\{\begin{array}{ll}
\min \; \; \; f(x) \\[4pt]
\mbox{s.t. } \quad g(x) = 0, \quad x\in C,
\end{array}
\right.
\end{eqnarray}
where $f : X \mapsto \mathbb{R}\cup\lbrace +\infty\rbrace$ is function, $g : X\mapsto Y$ is a mapping from a (real) Banach space $(X, \Vert \cdot\Vert_X)$ to a (real) Banach space $(Y,  \Vert \cdot\Vert_Y)$, and $C$ is a  nonempty closed subset of $X$.  

Suppose that $x_0$ is a local solution to problem \eqref{opt}. Let us state now our basic assumptions that will allow us to establish first order optimality conditions at $x_0$. 
\begin{itemize}
\item[${\bf (H_f)}$]$f$ is G\^ateaux differentiable at $x_0$ and locally Lipschitz around $x_0$ with constant $K_f>0$, that is, there exists $r>0$ such that
$$f(x) - f(x') \leq K_f\Vert x - x'\Vert_X \quad \forall \;  x, x'\in B_X(x_0, r).$$
\item[${\bf (H_g)}$] $g$ is  G\^ateaux differentiable at $x_0$.
\end{itemize}

If ${\bf (H_g)}$ holds true, we will denote by $Dg(x_0): X \to Y$ the G\^ateaux derivative and by $D^{\ast}g(x_0): Y^{\ast} \to X^{\ast}$ its adjoint operator. Similar notations will be used for the G\^ateaux derivative of $f$ if ${\bf (H_f)}$ holds.  

We recall that system
\begin{eqnarray}\label{system}
 x\in C \hspace{0.3cm} \mbox{and}  \quad g(x) = 0,
\end{eqnarray}
is said to be  {\it calm} at  $x_0 \in g^{-1}(0) \cap C$ if there exist $a > 0$ and $s > 0$ such that
\begin{eqnarray}\label{calm}
d_{g^{-1}(0)\cap C}(x) \leq a \Vert g(x)\Vert_Y \hspace{0.5cm}\forall \; x\in B_X(x_0, s)\cap C.
\end{eqnarray}

The following result gives existence of Lagrange multipliers for  problem (\ref{opt}) under the calmness condition  \eqref{calm} and the weak differentiability assumptions ${\bf (H_f)}$-${\bf (H_g)}$.
\begin{theorem}\label{thm2}  Suppose that ${\bf (H_f)}$-${\bf (H_g)}$ hold and that system {\rm \eqref{system}} is calm  at $x_0$. Let $K_f$ and $a$ be as in ${\bf (H_f)}$ and {\rm\eqref{calm}},  respectively.   Then, 
\begin{itemize}
\item[{\rm$(i)$}] if $x_0 \in \mbox{{\rm int}}(C)$, then there exists $y^*\in Y^*$, with $\Vert y^*\Vert_{Y^{\ast}} \leq K_f a$,   such that
$$Df(x_0) + D^*g(x_0)y^* = 0.$$
\item[$(ii)$] If  $g$ is locally Lipschitz around $x_0$ with constant $K_g>0$, then
\begin{equation}\label{maximizationh}Df(x_0)h + K_fa\Vert Dg(x_0)h\Vert_{ Y } + K_f(1+K_{g}a)  d^-_C(x_0; h) \geq 0 \quad \forall \; h\in X {\color{blue}.}\end{equation}
In particular,  there exists $y^*\in Y^*$, with $\Vert y^*\Vert_{Y^{\ast}} \leq K_f a$, such that
$$0 \in Df(x_0) + D^*g(x_0)y^* +  N_{C}(x_0).$$
If, in addition,   $K(C, x_0)$ is convex then there exists $y^*\in Y^*$, with $\Vert y^*\Vert_{Y^{\ast}} \leq K_f a$, such that
$$0 \in Df(x_0) + D^*g(x_0)y^* + (K(C,x_0))^0.$$
\end{itemize}
\end{theorem}

\begin{proof} Since $x_0$ is a local solution of problem (\ref{opt}) and  $f$ satisfies ${\bf(H_f)}$, by \cite[Proposition 2.4.3]{Clarke90} we have that $x_0$ is  a local  minimum of
$$x\in X \mapsto f(x) + K_f  d_{g^{-1}(0) \cap C}(x).$$
Using the calmness assumption  of  system (\ref{system}), we get that $x_0$ is a local solution  to
\begin{equation}\label{auxiliaryproblem}
\mbox{min } \;   f(x) + K_fa \Vert g(x)\Vert_Y \;  \; \; \mbox{s.t. } \; \; x\in C.
\end{equation}
Now, let us prove assertion ${\rm(i)}$. Since $x_0 \in \mbox{int}(C)$, 
there exists $s>0$ such that
$$f(x) + K_fa \Vert g(x)\Vert_Y \geq f(x_0) \quad \forall \; x\in B_{X}(x_0, s).$$
Let $h\in X$ be arbitrary and choose $\tau >0$  small enough such that $x_0+\tau h \in B_{X}(x_0, s)$. Then
$${{f(x_0+\tau h)- f(x_0)}\over \tau} + K_fa \left\| {{g(x_0+\tau h)-g(x_0)}\over \tau}\right\|_Y \geq 0.$$
Using that $f$ and $g$ are G\^ateaux differentiable  at $x_0$, we get
$$Df(x_0)h + K_fa\Vert Dg(x_0)h\Vert_{Y} \geq 0.$$
This means that the convex function $h\mapsto Df(x_0)h + \Vert Dg(x_0)h\Vert_{Y}$  attains its minimum at $h=0$. Thus, the (convex) subdifferential calculus produces  a  $y^* \in Y^*$, with $\Vert y^*\Vert_{Y^{\ast}} \leq K_fa$,  such that
$$Df(x_0) + D^*g(x_0)y^* = 0.$$

In order to prove assertion ${\rm(ii)}$, note that  since $x_0$ solves locally \eqref{auxiliaryproblem} and $f$ and $g$ are locally Lipschitz at $x_0$, by using   \cite[Proposition 2.4.3]{Clarke90} again, we obtain the existence of  $s>0$ such that
$$f(x) + K_fa \Vert g(x)\Vert_Y + K_f(1+K_{g}a) d_C(x) \geq f(x_0) \quad \forall \;  x\in B_{X}(x_0, s).$$
 Let $h\in X$ be arbitrary and choose a sequence $\tau_n \to 0^+$ such that
 $$d^-_C(x_0  ;  h) = \lim_{n\to +\infty} {{d_C(x_0+\tau_nh)}\over {\tau_n}}.$$
 Then, using the G\^ateaux differentiability of $f$ and $g$, we get 
\begin{equation}\label{inequality_with_d}
Df(x_0)h + K_fa\Vert Dg(x_0)h\Vert_{Y} + K_f(1+K_{g}a)  d^-_C(x_0; h) \geq 0 .\end{equation}
Noting that $d^-_C(x_0;h) \leq d^{\circ}_{C}(x;h)$,
  we obtain
$$Df(x_0)h + K_fa\Vert Dg(x_0)h\Vert_{Y} + K_f(1+ K_ga) d^{\circ}_C(x_0 ; h) \geq 0\quad  \forall \; h\in X,$$
or equivalently the convex function
$$h\in X \mapsto Df(x_0)h + K_fa\Vert Dg(x_0)h\Vert_{Y} + K_f(1+ K_ga) d^{\circ}_C(x_0, h)$$
attains its minimum at $h=0$.  Using that $\partial_{C} d^{\circ}_C(x_0, \cdot)(0)= \partial_{C} d_{C}(x_0)$ and \eqref{normalconeanddistance}, the (convex) subdifferential calculus produces a $y^* \in Y^*$, with $\Vert y^*\Vert_{Y^{\ast}}  \leq K_fa$, such that
$$-Df(x_0) - D^*g(x_0)y^* \in  K_f(1+K_ga)\partial d_C(x_0) \subset  N_{C}(x_0).$$
So that assertion {\rm(ii)} follows. 
Finally, inequality \eqref{inequality_with_d} yields 
$$
Df(x_0)h + K_fa\Vert Dg(x_0)h\Vert_{Y} \geq 0 \quad  \forall \; h\in K(C,x_0).
$$
Thus, if $ K(C,x_0)$ is convex, the last assertion in {\rm(ii)} follows from the convex subdifferential calculus. \hfill \end{proof} 

\vskip 0.5cm

Now consider the following optimization problem
\begin{eqnarray}\label{opt1}
\left\{\begin{array}{ll}
\min \; \; \; f(x) \\[4pt]
\mbox{s.t. } \quad g(x)  \in D, \quad x\in C,
\end{array}
\right.
\end{eqnarray}
\\
and the system
\begin{eqnarray}\label{system1}
\mbox{Find} \quad x\in C, \quad g(x) \in D.
\end{eqnarray}

System \eqref{system1} is said to be calm at $x_0 \in g^{-1}(D)\cap C$ if there exist $a > 0$ and $s > 0$ such that
\begin{eqnarray}\label{calm1}
 \quad  d_{g^{-1}(D)\cap C}(x) \leq a d_D(g(x)) \hspace{0.4cm} \forall \; x\in B_X(x_0, s)\cap C.
\end{eqnarray}

Problem (\ref{opt1}) can be rephrased as follows 

\begin{eqnarray}\label{opt2}
\left\{\begin{array}{ll}
\min \; \; \; \tilde{f}(x, y) \\[4pt]
\mbox{s.t. } \quad \tilde{g}(x, y) = 0, \quad (x, y) \in C\times D,
\end{array}
\right.
\end{eqnarray}
where $\tilde{f}(x, y) = f(x)$ and $\tilde{g}(x, y) = g(x) - y.$  Therefore,  (\ref{opt1})  can be written in the form (\ref{opt}). In the following result, we transfer the calmness property of system (\ref{system1}) to that of system
\begin{eqnarray}\label{system2}
\mbox{Find} \quad (x, y) \in C\times D, \quad \tilde{g}(x, y) = 0,
\end{eqnarray}
where the product space $X\times Y$ is endowed with the norm given by the sum of the norms in $X$ and $Y$.

\begin{lemma}\label{lemmaS} Suppose that $g$ is locally Lipschitz around $x_0$ and set $y_0 := g(x_0)$. Then, the following assertions are equivalent:
\begin{itemize}
\item[{\rm(i)}] The system {\rm(\ref{system1})} is calm at $x_0 \in g^{-1}(D)\cap C$.

\item[{\rm(ii)}] The system {\rm(\ref{system2})} is calm at $(x_0, y_0) \in C\times D$.

\end{itemize}
\end{lemma}
\begin{proof}  For notational convenience, we omit the subscripts for the norms $\|\cdot\|_{X}$ and  $\|\cdot\|_{Y}$. 
 ${\rm(i)} \Rightarrow {\rm(ii)}$: Since the system (\ref{system1}) is calm at $x_0 \in g^{-1}(D)\cap C$ and $g$ is locally Lipschitz around $x_0${\color{blue},} there exist $ a> 0$, $s>0$ and $K_g>0$ such that 
$$   d_{g^{-1}(D)\cap C}(x) \leq a d_D(g(x)) \hspace{0.5cm} \forall \; x\in B_X(x_0, 3s)\cap C, $$
and
$$  \Vert g(x) - g(x')\Vert \leq K_g\Vert x-x'\Vert \hspace{0.5cm} \forall \; x, \; x'\in B_X(x_0, 3s).$$
Let $(x, y) \in B((x_0, y_0), s)\cap \left( C\times D\right)$. For all $t \in ]0, s[$ there exists $u\in g^{-1}(D)\cap C$ such that 
$$\Vert x-u\Vert \leq  d_{g^{-1}(D)\cap C}(x)  + t \leq \Vert x-x_0\Vert + t \leq 2s,$$
and this asserts that $u\in B(x_0, 3s)\cap \left(g^{-1}(D) \cap C\right)$. Thus,
$$ \Vert g(x) - g(u)\Vert  \leq K_g\Vert x-u\Vert. $$
We have 
\begin{equation}\label{inequality_d_tilde}
d_{\tilde{g}^{-1}(0)\cap  (C\times D)}(x, y) = \inf_{v\in C\cap g^{-1}(D)} [\Vert x -v \Vert + \Vert y-g(v)\Vert ] \leq \Vert x -u\Vert + \Vert y-g(u)\Vert,
\end{equation} and using the triangle inequality, we get 
\begin{align*}
 \Vert y-g(u)\Vert  + \Vert x -u\Vert & \leq \Vert y-g(x)\Vert + \Vert g(x)-g(u)\Vert + \Vert x -u\Vert\\[4pt]
 &  \leq  \Vert y-g(x)\Vert + (1+ K_g)\Vert x-u\Vert\\[4pt]
 &  \leq \Vert y-g(x)\Vert +  (1+ K_g) d_{g^{-1}(D)\cap C}(x)  + t(1+ K_g)\\[4pt]
 & \leq \Vert y-g(x)\Vert +  (1+ K_g) a \Vert y - g(x)\Vert  + t(1+ K_g)\\[4pt]
 &\leq (1+a(1+K_g)) \Vert y - g(x)\Vert  + t(1+ K_g)\\[4pt]
& = (1+a(1+K_g)) \Vert \tilde{g}(x, y)\Vert  + t(1+ K_g).
\end{align*}
As $t$ is arbitrary,  relation \eqref{inequality_d_tilde} yields
$$\forall \; (x, y) \in B((x_0, y_0), s)\cap  (C\times D), \quad d_{\tilde{g}^{-1}(0)\cap  (C\times D)}(x, y) \leq  (1+a(1+K_g)) \Vert \tilde{g}(x, y)\Vert,$$
which implies that {\rm(ii)} holds. The implication  ${\rm(ii)} \Rightarrow {\rm(i)}$ is obvious since the following inequality holds true for all $x\in X$ and $y\in Y$
$$d_{\tilde{g}^{-1}(0)\cap  (C\times D)}(x, y) \geq d_{g^{-1}(D)\cap C}(x).$$
\end{proof} 

The following theorem, which is a consequence of Theorem \ref{thm2}, Lemma \ref{lemmaP} and Lemma  \ref{lemmaS},  gives the existence of KKT multipliers for  problem (\ref{opt1}) under the calmness condition and the weak differentiability assumptions ${\bf(H_f)}$-${\bf(H_g)}$.
\begin{theorem}\label{thm2_b}  Let $x_0$  be a local solution  to   problem  {\rm\eqref{opt1}}  and suppose that system {\rm \eqref{system1}} is calm  at $x_0$. Suppose that ${\bf(H_f)}$ and ${\bf(H_g)}$ hold and  that $g$ is locally Lipschitz around $x_0$. Then
\begin{itemize}
\item[{\rm(i)}] There exists $y^*\in N_{D}(g(x_0))$, with $\Vert y^*\Vert_{Y^{\ast}}  \leq K_f  (1+a(1+K_g))$  {\rm(}where $K_f$, $K_g$ and $a$ are as in ${\bf(H_f)}$, ${\bf(H_g)}$  and {\rm(\ref{calm1})}, respectively{\rm)}, such that
$$ -Df(x_0) - D^*g(x_0)y^* \in   N_{C}(x_0).$$
\item[{\rm(ii)}] Moreover, if $K(C, x_0)$ and $K(D, g(x_0))$ are convex and $C$  is tangentially regular at $x_0$ or   $D$ is tangentially regular at $g(x_0)$, then there exists $y^*\in (K(D, g(x_0)))^0$  such that  $\Vert y^*\Vert_{Y^{\ast}}  \leq K_f (1+a(1+K_g))$ and 
$$0\in Df(x_0) +D^*g(x_0)y^* + (K(C, x_0))^0.$$
\end{itemize}
\end{theorem}
\begin{proof}
Since $x_0$ solves \eqref{opt1} locally, $(x_0, g(x_0))$ is a local solution to problem \eqref{opt2}. Using that the constant $a$ satisfies \eqref{calm1}, the proof of Lemma \ref{lemmaS} shows that the calmness constant associated to system \eqref{system2} is given by $(1+a(1+K_g))$. Applying the second assertion in Theorem \ref{thm2}{\rm(ii)} to problem  \eqref{opt2},  yields the first assertion {\rm(i)}. In order to prove assertion {\rm(ii)}, note that \eqref{maximizationh}  implies that 
$$Df(x_0)h + K_f (1+a(1+K_g))\Vert Dg(x_0)h-k\Vert_{Y} \geq 0 \quad \forall \; (h,k) \in K(C\times D, (x_0,g(x_0)).$$
By Lemma \ref{lemmaP} we have that $K(C\times D, (x_0,g(x_0))= K(C,x_0) \times K(D, g(x_0))$, which is a convex set. The result then follows from standard convex analysis calculus. 
\end{proof}
\section{Metric regularity  under G\^ateaux differentiability}\label{metric_regularity}


In this section, we first provide a sufficient condition for a stronger property than the calmness of system \eqref{system}, namely its metric regularity (see \cite{Dontchev_Rockafellar_metric_regularity,Ioffe_book} and the references therein). Then, and as in the previous section, we deduce the corresponding sufficient condition for system \eqref{system1} by reducing it to an instance of system \eqref{system} (see \eqref{system2}).   

For system \eqref{system}, the sufficient condition is given by the {\it constraint qualification} ${ \bf(H_{cq})}$ below. In the remainder of this article, given a subset $A$ of a real Banach space $(Z, \| \cdot \|_{Z})$, $y\in A$ and $r>0$, we set $B_{A}(y,r):= B_{Z}(y,r) \cap A$.  \smallskip 

Throughout this section, we fix a point $x_0 \in g^{-1}(0) \cap C$. We consider the following constraint qualification condition on a neighbourhood of $x_0$. 
\begin{itemize}
\item[${ \bf(H_{cq})}$] there exist $\alpha > 0$ and $r>0$ such that $g$ is continuous and G\^ateaux differentiable on $B_{C}(x_0,r)$ and 
\begin{eqnarray}\label{open}
B_Y(0, 1) \subset Dg(x)\big(B_{ K(C,x)}(0, \alpha) \big)\quad \forall \;  x\in  B_{C}(x_0,r).
\end{eqnarray}
\end{itemize}
\begin{remark} For each $x\in B(x_0,r)$ consider a right-inverse $G(x): Y \rightrightarrows X$ of $Dg(x)$, i.e. $Dg(x)G(x)y= \{y\}$ for all $y\in Y$  {\rm(}we know that such right-inverse exists because \eqref{open} implies that $Dg(x)$ is surjective{\rm)}. Then, assumption \eqref{open} can be rephrased in terms of $G$ as follows
$$
\sup_{x\in B_{X}(x_0,r), \;  y\in B_{Y}(0,1)} \; \;  \inf_{v \in G(x)y \cap K(C,x)}\|v\|_{X} \leq \alpha. 
$$

\end{remark}
The main result of our article is the following.
\begin{theorem}\label{thm3} Suppose that  $({\bf H_g})$ and ${ \bf(H_{cq})}$ hold true and let $\alpha>0$ and $r>0$ be such that \eqref{open} is satisfied.  Then, for all $r_1>0$ and $r_2>0$, with $r_1+r_2=r$, and all 
$$(x, y) \in D_{r_1,r_2}:=\left\{ (u, v) \in  B_{C}(x_0,r_1)\times Y: \, \Vert g(u)-v\Vert_{Y} < {{r_2}\over\alpha}\right\},$$
we have
\begin{eqnarray}\label{regm}
d_{g^{-1}(y)\cap C}(x) \leq \alpha\Vert g(x) - y\Vert_Y.
\end{eqnarray}
\end{theorem}
\begin{proof}  The proof  is inspired from  \cite{Ekeland_inverse}.  Fix $(x, y) \in D_{r_1,r_2}$.  If $y=g(x)$ then  \eqref{regm} is trivial, so let us assume that  $y\neq g(x)$. Consider the function $h: X\mapsto \mathbb{R}$ defined as
$$h(u) := \Vert g(u) - y\Vert_{Y}.$$
Let $\beta > \alpha$  be such that $0<h(x)=\Vert g(x) - y\Vert_{Y} < {{r_2}\over\beta}$. As $h$ is continuous and bounded from below on the closed set $B_C(x_0, r)$ and, evidently,  
$$h(x)\leq  \inf_{x' \in B_{C}(x_0,r)} h(x') +   h(x),$$
  Ekeland's variational principle (see \cite[Theorem 1.1]{Ekeland_variational_principle}) gives the existence of  $\bar{u}\in B_C(x_0, r)$ such that
\begin{align}\label{ineq1}
h(\bar{u})  & \leq h(x),
\\[6pt]
\label{ineq2}
\Vert \bar{u} - x\Vert_{X} &\leq  \beta h(x),
\\[1pt]
\label{ineq3}
h(\bar{u}) & \leq h(u) + {1\over \beta}\Vert \bar{u} - u\Vert_{X}\quad  \forall  \; u\in B_C(x_0, r).
\end{align}
Inequality \eqref{ineq2}  and the choice of $x$ and $\beta$  imply that 
\begin{eqnarray}\label{ineq4}
\Vert \bar{u} - x\Vert_{X} < r_2 \; \;  \hbox{  and so } \; \;  \Vert \bar{u} - x_0\Vert_{X} \leq \Vert \bar{u}-x\Vert_{X}+\Vert x-x_0\Vert_{X} < r_{2}+r_{1}=r.
\end{eqnarray}
{\it Claim: we have that $y = g(\bar{u})$}. Let us assume for a moment that the claim is true. By (\ref{ineq2}), we obtain
 $$d_{g^{-1}(y)\cap C}(x) \leq \beta\Vert g(x) - y\Vert_Y,$$
 and, as $\beta>\alpha$ is arbitrary, we get  that    (\ref{regm}) holds true. 
 
It remains to prove the claim.   Suppose the contrary and define 
$$w=\frac{y-g(\bar{u})}{\|y-g(\bar{u}) \|_{Y}}.$$ 
Since  $\bar{u}\in B_C(x_0, r)$,  assumption ${ \bf(H_{cq})}$ implies the existence of  $v \in B_{K(C,\bar{u})}(0,\alpha)$ such that 
$$
w= Dg(\bar{u})v.
$$
Since  $v \in B_{K(C,\bar{u})}(0,\alpha)$, there exist sequences $\tau_n \to 0^+$ and $v_n\to v$ such that 
$$u_n:= \bar{u}+ \tau_n v_n \in C \hbox{  for $n$ sufficiently large}.$$
We may write  $u_n = \bar{u} + \tau_n   v + o(\tau_n) \in  C$, where $\displaystyle \lim_{n\to +\infty}{{o( \tau_n )}\over { \tau_n}} = 0$. Note that the second inequality in  (\ref{ineq4}) implies that $u_n \in B_C(x_0, r)$ for $n$ sufficiently large.
Now, using inequality (\ref{ineq3}), we get
\begin{equation}\label{inequalityforh}h(\bar{u}) \leq h(u_n) + {1\over \beta}\Vert  \tau_n  v+ o(\tau_n)\Vert_{X}.
\end{equation}
On the other hand, since $g$ is G\^ateaux differentiable at $\bar{u}$, we have
$$g(u_n) = g(\bar{u}) +  \tau_n Dg(\bar{u})v +  \tau_n \varepsilon( \tau_n ), \hbox{  where  } \lim_{n\to +\infty}\varepsilon( \tau_n ) = 0,$$
which, combined with \eqref{inequalityforh}, ensures that
$${{\Vert g(\bar{u}) - y +  \tau_n Dg(\bar{u})v +  \tau_n   \varepsilon( \tau_n )\Vert_{Y} -  \Vert g(\bar{u}) - y\Vert}_{Y}\over { \tau_n }} \geq -{1\over \beta}\left\Vert v+{{o( \tau_n )}\over { \tau_n }} \right\Vert_{X}.$$
Since 
$$
\lim_{n\to +\infty}{{\Vert g(\bar{u}) - y + \tau_n  Dg(\bar{u})v \Vert_{Y} -  \Vert g(\bar{u}) - y\Vert}_{Y}\over { \tau_n }}= \max_{ y^{\ast}  \in \partial \| \cdot \|_{ Y }(g(\bar{u}) - y)} \langle y^{\ast}, Dg(\bar{u})v\rangle_{Y},
$$
we get the existence of $y^{\ast}_{v} \in  \partial \| \cdot \|(g(\bar{u}) - y)$, such that 
\begin{equation}\label{ineq5}
-1=\langle y_v^{\ast}, w\rangle_{Y}=\langle y_v^{\ast}, Dg(\bar{u})v\rangle_{Y}  \geq   -{1\over \beta}\Vert v\Vert_{X}\geq -\frac{\alpha}{\beta},
\end{equation}
where the first equality follows from  the fact that we are assuming that $g(\bar{u}) \neq y$ and the standard relation 
$$y^{\ast}_{v} \in  \partial \| \cdot \|_{Y}(g(\bar{u}) - y) \Leftrightarrow \; \Vert y_v^{\ast} \Vert_{Y^\ast}=1 \hspace{0.2cm} \mbox{and} \hspace{0.2cm} \langle  y_v^{\ast},g(\bar{u}) - y\rangle_Y = \Vert g(\bar{u}) - y\Vert_{Y}. \hspace{0.2cm}    $$

Since \eqref{ineq5} contradicts $\alpha<\beta$, the claim follows. 
\end{proof} \vspace{0.2cm}

The  previous result  extends the following inverse function theorem result, proved first in \cite[Theorem 2]{Ekeland_inverse} in the case $C=X$. 
\begin{corollary}\label{cor2} Suppose that the  assumptions of Theorem \ref{thm3} are satisfied. Then,
\begin{equation}\label{metric_regularity_x_0_fixed}
d_{ g^{-1}(y)\cap C}(x_0) \leq \alpha\Vert y\|_{Y}  \quad \forall \;  y\in Y, \hbox{  with  } \Vert y\Vert_Y < {r\over \alpha}.
\end{equation}
Consequently, for all $y\in Y$, with $\Vert y\Vert_Y < {r\over \alpha}$,  and for all $\beta > \alpha$ there exists $x\in  g^{-1}(y)\cap C$   such that
\begin{equation}\label{two_distances_to_x_0}
\Vert x-x_0\Vert_X < r, \quad \Vert x-x_0\Vert_X \leq \beta \Vert y\Vert_Y.
\end{equation}
\end{corollary}
\begin{proof} By Theorem \ref{thm3}, in order to prove \eqref{metric_regularity_x_0_fixed} it suffices to choose $\varepsilon>0$ such that $(x_{0},y) \in D_{\varepsilon, r-\varepsilon}$, which is possible because of the strict inequality in  \eqref{metric_regularity_x_0_fixed}. It remains to prove that \eqref{two_distances_to_x_0} holds for $\beta>\alpha$ and $\| y\|_{Y} <  r/\alpha$. In this case, the first inequality in \eqref{metric_regularity_x_0_fixed} becomes strict and we get the existence of $x_{\beta} \in   g^{-1}(y)\cap C$   such that the second inequality in  \eqref{two_distances_to_x_0} holds true.

Since there exists $\varepsilon>0$ such that $\|y\|_{Y} \leq (r-\varepsilon)/\alpha$ then the first inequality in \eqref{two_distances_to_x_0}  holds for $x_\beta$ provided that  $\alpha <\beta <\alpha r/(r-\varepsilon)$. If $\beta \geq \alpha r/(r-\varepsilon)$ then \eqref{two_distances_to_x_0} holds for $x_{\beta'}$ with $\beta' \in ]\alpha, \alpha r/(r-\varepsilon)[$ and so $\|x_{\beta'}-x_0\|_{X} \leq \beta'\|y\|_{Y} \leq \beta \|y\|_{Y}$. The result follows. 
\end{proof} \vspace{0.2cm}

Now, we study the corresponding metric regularity property for system \eqref{system1}.  We consider the following qualification condition: 
\begin{itemize}
\item[${ \bf(H_{cq}')}$] there exist $\alpha_1, \alpha_2> 0$ and $r>0$ such that $g$ is continuous and G\^ateaux differentiable on $B_{C}(x_0,r)$ and  
\begin{equation}\label{open1}
\begin{array}{l}
B_Y(0, 1) \subset Dg(x)\big(B_{K(C,x)}(0, \alpha_1) \big)-B_{K(D,y)}(0, \alpha_2)\\[6pt]
\hspace{3cm} \forall \; (x,y) \in B_{C \times D}( (x_0,g(x_0)), r).
\end{array}
\end{equation}
\end{itemize}  
\begin{theorem}\label{casDconvexe} Suppose that $({\bf H_f})$, $({\bf H_g})$ and ${ \bf(H_{cq}')}$ hold true and that at least one of the sets $C$ and $D$ is convex.  Denote $\alpha= \max\{\alpha_1,\alpha_2\}$. Then, for all $r_1>0$ and $r_2>0$, with $r_1+r_2=r$, and all 
$$(x, y) \in D_{r_1,r_2}:=\left\{ (u, v) \in  B_{C}(x_0,r_1)\times Y: \, d_{B_{D}(g(x_{0}), r_1)}(g(u)-v) < {{r_2}\over\alpha}\right\},$$
we have
$$
d_{g^{-1}(D+y)\cap C}(x) \leq \alpha  d_{B_{D}(g(x_{0}), r_1)}(g(x)-y).
$$
\end{theorem}
\begin{proof} 
Using that  at least one of the sets $C$ and $D$ is convex, for all $(x',y')\in C\times D$ we have
$$
  B_{K(C,x')}(0, \alpha_1)\times B_{K(D,y')}(0,\alpha_2) \subseteq B_{K(C\times D,(x',y'))}((0, 0), \alpha). 
$$
Therefore, defining $\tilde{g}: X\times Y \to Y$ as $\tilde{g}(x,z):= g(x)-z$,  condition \eqref{open1} implies that  
\begin{equation}\label{qualification_condition_augmented}
B_{Y}(0,1) \subseteq D \tilde{g}(x',y')\left[ B_{K(C\times D,(x',y'))}((0, 0), \alpha)  \right] \; \; \forall \; (x',y') \in B_{C \times D}( (x_0,g(x_0)), r).
\end{equation}

Now, let $(x,y) \in D_{r_1,r_2}$  and $\varepsilon>0$ be such that $d_{B_{D}(g(x_{0}), r_1)}(g(x)-y)+\varepsilon <{{r_2}\over\alpha} $.  Then, there exists $z_{\varepsilon}\in B_{D}(g(x_{0}), r_1)$ such that 
\begin{equation}
\|g(x)-y-z_{\varepsilon}\|_{Y}\leq d_{B_{D}(g(x_{0}), r_1)}(g(x)-y)+\varepsilon <{{r_2}\over\alpha}.
\end{equation}
By \eqref{qualification_condition_augmented}, we can apply Theorem \ref{thm3} to $ \tilde{g}$ and  deduce that 
\begin{equation}\label{composed_inequality}
d_{\tilde{g}^{-1}(y) \cap \left(C \times D \right)}(x,z_{\varepsilon}) \leq  \alpha \|g(x)-z_{\varepsilon}-y\|_{Y}\leq \alpha d_{B_{D}(g(x_{0}), r_1)}(g(x)-y)+\alpha \varepsilon.
\end{equation}
Finally, since  $(x',z') \in \tilde{g}^{-1}(y) \cap \left(C \times D \right)$ iff $x' \in C$, $z' \in D$ and $g(x')-y=z'$, we get that 
\begin{equation}\label{composed_inequality1}
d_{g^{-1}(D+y)\cap C}(x)\leq  d_{\tilde{g}^{-1}(y) \cap \left(C \times D \right)}(x,z_{\varepsilon}) .
\end{equation}
Since $\varepsilon$ is arbitrary, the result follows from \eqref{composed_inequality}-\eqref{composed_inequality1}
\end{proof}

\vskip 0.5cm

We can ask if we can replace the assumption ${ \bf(H_{cq}')}$ by the following one 

\begin{itemize}
\item[${ \bf(H_{cq}'')}$] there exist $\alpha_1, \alpha_2> 0$ and $r>0$ such that $g$ is continuous and G\^ateaux differentiable on $B_{C}(x_0,r)$ and  
\begin{equation}\label{open1_restricted}
\begin{array}{l}
B_Y(0, 1) \subset Dg(x)\big(B_{K(C,x)}(0, \alpha_1) \big)-B_{K(D,g(x))}(0, \alpha_2) \hspace{0.4cm}\forall \; x \in B_{g^{-1}(D)\cap C}( x_0, r).
\end{array}
\end{equation}
\end{itemize}  
As the following example shows, the answer is negative.
\begin{example}\label{contrexample} Let $C$ and $D$ be closed sets in $\mathbb{R}^2$ defined by 
$$C = \{ (x, y) \in \mathbb{R}^2 : \,  x\geq 0, \, x^2+(y+1)^2 = 1 \},$$
and 
$$ D= \{ (x, y) \in \mathbb{R}^2 : \, [  y = x] \hbox{  {\rm or}  } [ x\geq 0, \, x^2+(y+2)^2 = 4] \},$$ 
{\rm(}see {\rm Figure \ref{first_picture}}{\rm)} and take $g$ be the identity function in  $\mathbb{R}^2$. Then $C \cap D = \{0\}$,  $g^{-1}(C \cap D) = \{0\}$, $K(C, (0,0)) = \mathbb{R}_+ \times \{0\} $ and $K(D, (0, 0)) = \{(x, x) : \, x\in \mathbb{R}\}\cup \left(\mathbb{R}_+ \times \{0\}\right).$ Thus, 
$$B_{\mathbb{R}^2}(0, 1) \subset B_{K(C, (0,0))}(0, 2) - B_{K(D, (0,0))}(0, 2).$$
 Similarly, we have that \eqref{open1_restricted} holds true and it is easy to check that \eqref{open1} does not hold. We will show that  there is no $a > 0$ such that 
$$d_{g^{-1}(C \cap D)}(u) \leq a d(g(u), D) \quad \hbox{ for $u \in C$ near $0$}.$$
 Indeed,  for $x>0$ and $x^2+(y+1)^2 = 1$, with $(x, y)$ near $(0, 0)$, we have 
$$d_{g^{-1}(C \cap D)}(x, y) = \sqrt{x^2+y^2} \hbox{  and  } d(g(x, y), D) \leq 2 - \sqrt{ 4 -(x^2+y^2)}$$
and the inequality 
$$\sqrt{x^2+y^2} \leq a (2 - \sqrt{2-(x^2+y^2)}) \approx  a{{x^2+y^2}\over  4 }$$
is never satisfied when $(x, y)$ is sufficiently near to $(0, 0)$. 
\end{example}

\begin{figure}[htp]
\begin{center}
  \includegraphics[width=3in]{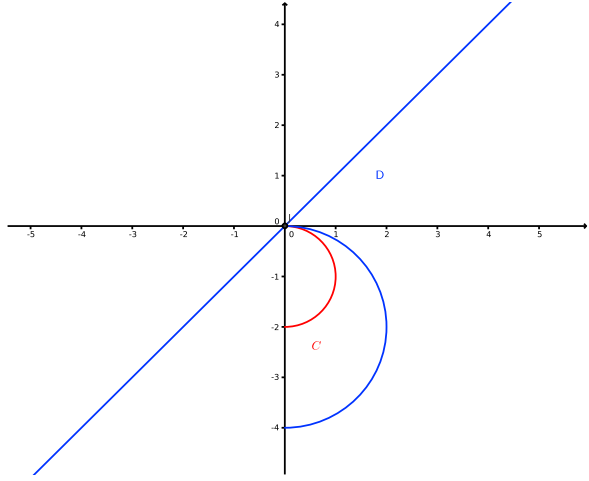}
\caption{Sets $C$ and $D$ in Example \ref{contrexample}.}
\label{first_picture}
\end{center}
\end{figure}

 \section{Application to stochastic optimal control in continuous time}\label{section_stochastic_control_continuous_time}

Let  $T>0$ and consider a filtered probability space  $(\Omega, \F, \mathbb{F}, \mathbb{P})$, on which a $d$-dimensional ($d \in \mathbb{N}^{*}$)     Brownian motion $W(\cdot)$  is defined. We suppose that $\mathbb{F}=\left\{\F_{t}\right\}_{0\leq t \leq T}$  is the   natural filtration, augmented by all $\mathbb{P}$-null sets in $\F$, associated to $W(\cdot)$. The filtration $\mathbb{F}$ is right-continuous, i.e. $\F_{t}= \cap_{t<u\leq T} \F_{u}$ (see \cite[Chapter I, Theorem 31]{Protter-book}).  

Recall that a stochastic process $v:\Omega  \times [0,T]  \to \R^n$ is {\it progressively measurable }w.r.t. $\FF$ if for all $t \in [0,T]$ the application $ \Omega \times [0,t]   \ni (s,\omega) \mapsto v(\omega,s) \in \R^n$ is $  \F_{t}\times \mathcal{B}([0,t])$ measurable (here $\mathcal{B}([0,t])$ denotes the set of Borel sets in $[0,T]$). Let us define the space 
$$\begin{array}{ll} (L^{2,2}_{\FF})^{n}:=& \left\{ v \in L^{2}\left(\Omega ; L^{2}\left([0,T]; \R^{n}\right)\right) ; \ (\omega,t) \mapsto v(\omega,t):=v(\omega)(t) \  \right.\\[6pt]
\; &  \left. \; \; \hspace{4.4cm} \mbox{is progressively measurable}\right\}.
\end{array}$$ 
When $n=1$ we will simply denote $L^{2,2}_{\FF}:=(L^{2,2}_{\FF})^{1}$.  It is easy to see that $(L^{2,2}_{\FF})^{n}$, endowed with the scalar product
$$ \langle v_1, v_2 \rangle_{L^{2,2}} := \EE\left(\int_{0}^{T} v_{1}(t) \cdot v_{2}(t) \dd t \right), $$
is a Hilbert space. We denote by $\|\cdot\|_{2,2}:=\langle \cdot, \cdot \rangle_{L^{2,2}}^{\frac{1}{2}}$ the associated Hilbersian norm.  

In this section we consider the stochastic optimal control problem
$$\left.\begin{array}{l} \inf_{x,u} \; \EE\left( \int_0^T \ell(\omega,t,x(t),u(t)) \dd t + \Phi(\omega, x(T)) \right)\\[8pt]
\mbox{s.t. } \hspace{0.2cm}\dd x(t) = b(\omega,t,x(t),u(t)) \dd t + \sigma(\omega,t,x(t),u(t)) \dd W(t) \hspace{0.3cm} t\in (0,T), \\[6pt]
\hspace{1cm} x(0)=\hat{x}_0,\\[6pt]
\hspace{1cm} u \in \mathcal{\U},
\end{array}\right\} \eqno(SP)$$
where $\mathcal{\U}$ is a non-empty, closed subset of $(L^{2,2}_{\FF})^{m}$ and  $b: \Omega \times [0,T] \times \R^{n}\times \R^{m} \to \R^{n}$,  $ \sigma: \Omega \times [0,T] \times \R^{n}\times \R^{m} \to \R^{n\times d}$, $\ell: \Omega \times [0,T] \times \R^{n}\times \R^{m} \to \R$, $\Phi: \Omega \times \R^n \to \R$, and $\hat{x}_0 \in \R^n$ are given.  In what follows we use the notation $ b=(b^{i})_{1\leq i \leq n}$ and  $\sigma=(\sigma^{ij})_{1\leq i \leq n, \; 1\leq j \leq d }$, where each $b^{i}$ and $\sigma^{ij}$ is real valued. The columns of $\sigma$ are written $\sigma^{j}$ for $j=1,\hdots, d$. For $\psi=\ell$, $\Phi$, $b^{j}$, $\sigma^{ij}$ we will denote by $\nabla_{x}\psi$ the gradient of $\psi$ w.r.t. to $x$. We will also use the notation $b_x$ and $\sigma^j_x$ to denote, respectively, the Jacobians of $b$ and $\sigma^j$ w.r.t. $x$. Similar notations will be using when differentiating w.r.t. $u$.

In order to make problem $(SP)$ meaningful, we need to impose some assumptions on the data. Concerning the terms defining the dynamics $b$ and $\sigma$ we will assume
 \medskip\\
\textbf{(A1)}  For $\psi=b^{j}$, $\sigma^{ij}$ we have: \smallskip
\\
{\rm(i) } $\psi$ is $\F_{T} \otimes \mathcal{B}([0,T]\times \R^{n}\times \R^{m})$-measurable. \smallskip
\\
{\rm(ii) } For almost all (a.a.) $(\omega, t) \in \Omega \times  [0,T]$ the mapping $(x,u)\to \psi(\omega, t,x,u)$ belongs to   $C^{1}(\R^{n}\times \R^{m})$, the application  $(\omega, t) \in \Omega \times [0,T] \to \psi(\omega,t, \cdot, \cdot) \in C^{1}(\R^{n}\times \R^{m})$ is progressively measurable and there exists $c_1 >0$ and $\rho_1 \in L^{2,2}_{\FF}$ such that  almost surely (a.s.) in   $(\omega, t)$
\begin{equation}\label{acotamientoderivadas}\left\{\begin{array}{c} |\psi(\omega,t,x,u)|\leq  c_{1} \left(  \rho_{1}(\omega,t)+ |x|+|u|\right), \\[4pt]
																			|\nabla_x \psi (\omega, t,x,u)|+ |\nabla_u \psi(\omega,t,x,u)| \leq c_{1}. \\[4pt] 	
\end{array}\right.\end{equation}

Concerning the terms defining the cost functions $\ell$ and $\Phi$ we will assume \medskip\\
\textbf{(A2)} The functions  $\ell$ and  $\Phi$  are respectively  $\F_{T} \otimes  \mathcal{B}([0,T]\times \R^{n}\times \R^{m})$  and $\F_{T} \otimes \mathcal{B}(\R^{n})$ measurable. Moreover,  for a.a. $(\omega,t)$ the maps $(x,u)\to \ell(\omega, t,x,u)$ and     $x\to \Phi(\omega, x)$ are $C^1$. The application  $(\omega, t) \in \Omega \times [0,T] \to \ell(\omega,t, \cdot, \cdot) \in C^{1}(\R^{n}\times \R^{m})$ is progressively measurable. In addition, there exists $c_{2}>0$,  $\rho_2 \in L^{2,2}_{\FF}$ and $\rho_3 \in L^{2}(\Omega,\F_T)$ such that  almost surely in   $(\omega, t)$ we have 
\begin{equation}\label{crecimientoderivadascosto}\left\{\begin{array}{c}|\ell (\omega , t,x,u)|\leq  c_{2} \left(  \rho_2(\omega,t)+ |x|^2+|u|^2\right), \\[4pt]
 |\nabla_x\ell(\omega,t,x,u)|+ |\nabla_u \ell(\omega, t,x,u)|\leq  c_{2} \left(  \rho_2(\omega,t)+ |x|+|u|\right), \\[4pt]
 \ |\Phi (\omega, x)|\leq  c_{2} \left( \rho_3(\omega)+ |x|^{2}\right), \; |\nabla_x \Phi(\omega,x)|\leq  c_{2} \left( \rho_3(\omega)+ |x|\right). 
 \end{array}\right.\end{equation}

The previous assumptions are rather general and cover the case of linear quadratic problems {\rm(}see e.g. {\rm \cite[Chapter 3 and Chapter 6]{YongZhou})}. 

Our aim now is to provide a functional framework for problem $(SP)$ that will allow us to apply the abstract results in the previous sections to derive a first order optimality condition  at a local solution.  We proceed as in \cite{backhoffsilva14} and we focus first in writing the SDE constraint in the form of an equality constraint in a suitable function space. 
 
Let us consider the mapping $I: \R^{n} \times (L^{2,2}_{\FF})^{n} \times (L^{2,2}_{\FF})^{n\times d} \to (L^{2,2}_{\FF})^{n}$
\begin{equation}\label{amdmmrnnrnssasa}I(x_0,x_1,x_2)(\cdot):= x_0 + \int_{0}^{(\cdot)} x_{1}(s) \dd s + \sum_{j=1}^d \int_{0}^{(\cdot)} x_{2}^j(s) \dd W^j(s).\end{equation}
Standard results in It\^o's stochastic calculus theory  imply that $I$ is well defined.  Consider the {\it It\^o space} $\I^{n}:=I(  \R^{n} \times (L^{2,2}_{\FF})^{n} \times (L^{2,2}_{\FF})^{n\times d})$. Endowed with the scalar product
\begin{equation}\langle x,y \rangle_{\I^{n}}:=  x_{0} \cdot y_{0} + \EE\left(\int_{0}^{T}  x_{1}(t) \cdot y_{1}(t) \dd t \right) +  \sum_{j=1}^d \EE\left(\int_{0}^{T} x_2^{j}(t) \cdot y_2^j(t)\dd t \right),
\end{equation}
we have that $\I^{n}$ is a Hilbert space, which, since $I$ is injective (see \cite[Lemma 2.1]{backhoffsilva14}), can be identified with  $\R^{n} \times (L^{2,2}_{\FF})^{n} \times (L^{2,2}_{\FF})^{n\times d}$. Let us  denote by $\| \cdot \|_{\I^{n}}:= \langle \cdot, \cdot\rangle_{\I^{n}}^{\frac{1}{2}}$  the associated Hilbersian-norm.

%
Recall that by definition $x \in \I^{n}$ solves  the controlled SDE in $(SP)$ iff
\begin{equation}\label{rigorousdefsde}
x(t)= x_0 + \int_{0}^{t}b(s,x(s),u(s)) \dd s+ \int_{0}^{t} \sigma(s,x(s),u(s)) \dd W(s) \hspace{0.4cm} \forall \; t \in [0,T].
\end{equation}
It is well known that under  {\bf(A1)} equation \eqref{rigorousdefsde} admits a unique solution $x \in \I^{n}$ (see e.g. \cite[Chapter 5]{KaraShreve91}). It is also known that  $\EE\left( \sup_{t\in [0,T]} |x(t)|^{2}\right)$ is finite (see e.g. \cite[Lemma 2.2]{backhoffsilva14}).  A more precise information is given by the following lemma whose proof is by now standard. We provide here the details of the proof  since we need to obtain explicit expressions for the  involved constants.
\begin{lemma}\label{lemmacota1} For all $t\in [0,T]$ and $u\in (L^{2,2}_{\FF})^{m}$, the solution $x\in \I^{n}$ satisfies
\begin{equation}\label{cotal2} \EE\left( \sup_{s \in [0,t]} |x(s)|^{2} \right)= c \left[|x_0|^2+\EE\left(\int_{0}^{t}|b(s,0,u)|^{2} \dd s \right) + \EE\left(\int_{0}^{t}|\sigma(s,0,u)|^{2} \dd s \right)   \right],\end{equation}
where $c=\max\{24,6T\}e^{6Tc_1^2\max\{T, 4d\}}$.
\end{lemma}
\begin{proof} Using the inequality $(a_1+a_2+a_3)^{2} \leq 3(a_1^2+a_2^2+a_3^2)$ for all $a_1$, $a_2$ and $a_3$ in $\R$ and Jensen's inequality,   for all $0\leq s \leq t \leq T$ expression \eqref{rigorousdefsde} yields
$$
|x(s)|^{2} \leq 3\left( |x_0|^2 + s\int_{0}^{s}|b(s',x(s'), u(s'))|^{2} \dd s' + \left|\int_{0}^{s} \sigma(s',x(s'),u(s')) \dd W(s')\right|^{2}\right).
$$
By the linear growth  condition in \eqref{acotamientoderivadas} and the fact that $x\in \I^{n}$ and $u\in (L^{2,2}_{\FF})^m$, we have that $\sigma(\cdot,x(\cdot),u(\cdot)) \in (L^{2,2}_{\FF})^{n\times d}$ and so, for each $j=1,\hdots,d$, the $\R^n$-valued process $ s \in [0,T] \mapsto  \int_{0}^{s} \sigma^{j}(s',x(s'),u(s')) \dd W^{j}(s')$ is a martingale. Thus, defining $g(t):= \EE(\sup_{s\in [0,t]}|x(s)|^{2})$, Doob's inequality and the  Lipschitz  property of $b$ and $\sigma$ with respect to $x$ in \eqref{acotamientoderivadas} imply that
$$\begin{array}{rcl}
g(t)&\leq& 3\left[ |x_0|^2 + T\EE\left(\int_{0}^{t}|b(s,x(s), u(s))|^{2} \dd s\right) +4\EE\left(\int_{0}^{t}|\sigma(s,x(s), u(s))|^{2} \dd s\right)\right]\\[6pt]
\; & \leq & 3\left[ |x_0|^2 + 2T\EE\left(\int_{0}^{t}\left[|b(s,0, u(s))|^{2}+ c_{1}^{2} |x(s)|^{2}\right]  \right)\dd s\right. \\[6pt]
\; & \; & \hspace{0.5cm}\left.+ 8\EE\left(\int_{0}^{t}\left[|\sigma(s,0, u(s))|^{2}+ dc_{1}^{2} |x(s)|^{2}\right] \dd s\right)\right]\\[6pt]
\; & \leq & a+ b\int_{0}^{t} g(s) \dd s,
\end{array}
$$
where  
$$ a=\max\{24,6T\} \left[|x_0|^2+ \EE\left(\int_{0}^{t}|b(s,0, u(s))|^{2} \dd s\right) +\EE\left(\int_{0}^{t} |\sigma(s,0, u(s))|^{2} \dd s\right)\right], $$  
and $b= 6c_1^2\max\{T, 4d\}$. The result then follows from Gronwall's Lemma.
\end{proof}
\begin{remark} Estimates of the form \eqref{cotal2}  can be easily extended to any power $p>1$ by using in the previous proof the Burkholder-Davis-Gundy inequality {\rm(}see e.g. {\rm \cite{MR2340054})} instead of Doob's inequality.
\end{remark}

Now, let us consider the application $g: \I^{n} \times (L^{2,2}_{\FF})^{m} \to \I^{n}$ defined by
\begin{equation}\label{meqmemnnrnnrnr}g(x,u)(\cdot) :=  \hat{x}_0+  \int_{0}^{(\cdot)} b(s,x(s),u(s)) \dd s +  \int_{0}^{(\cdot)} \sigma(s,x(s), u(s)) \dd W(s)-x(\cdot),\end{equation}
which defines the SDE constraint in $(SP)$ by  imposing $g(x,u)=0$. Consider also the application $f: \I^{n} \times (L^{2,2}_{\FF})^{m} \to \R$ defined by
 $$f(x,u):=  \EE \left( \int_{0}^{T}\ell(t,x(t),u(t))\dd t + \Phi(x(T)) \right),$$
 which describes the cost functional in $(SP)$. Assumption \textbf{(A2)} implies that $f$ is well-defined. Problem $(SP)$ can thus be rewritten in the following abstract form 
$$ \inf \; \; f(x,u) \; \;  \mbox{subject to } \;  g(x,u)=0, \; \; \;  u\in \U. \eqno(SP)$$
We proceed now to verify that $f$ and $g$ satisfy  the assumptions considered in Section \ref{optimality_conditions_abstract_form}, when the underlying space  given by $X:= \I^n \times (L^{2,2}_{\FF})^{m}$.

We begin by studying some properties of  $g$. The following result is proved in the appendix in \cite{backhoffsilva14}. For the sake of completeness we provide here a short proof.
 \begin{lemma}\label{wrnwnrnwnrnqbqbqbq}  Under {\rm \textbf{(A1)}} the mapping $g$ is Lipschitz continuous and  G\^ateaux  differentiable.  Its  G\^ateaux derivative $Dg(x,u): \I^{n} \times (L^{2,2}_{\FF})^{m} \mapsto   \I^{n} $ is given by 
\begin{equation}\label{mmsnbabbasvvwwwaaass}\begin{array}{rcl} Dg(x,u)(z,v)(\cdot) &=&  \int_{0}^{(\cdot)}\left[   b_{x}(t,x(t),u(t)) z(t) +   b_{u}(t,x(t),u(t)) v(t)\right] \dd t\\[4pt]
 			\;     & \; & +\sum_{j=1}^d \int_{0}^{(\cdot)}\left[   \sigma_{x}^{j}(t,x(t),u(t)) z(t) +   \sigma_{u}^{j}(t,x(t),u(t)) v(t)\right] \dd W^j(t) \\[4pt]
			\; & \; & -z(\cdot),	\end{array}\end{equation} 
for all $(z,v) \in \I^{n} \times  (L^{2,2}_{\FF})^{m}$.
 \end{lemma}
 \begin{proof} Note that for any $(x,u_1)$, $(y,u_2) \in \I^{n}\times (L^{2,2}_{\FF})^{m}$ we have   
 $$\begin{array}{c}\|g(x,u_1)(\cdot)- g(y,u_2)(\cdot)\|_{\I^{n}}^{2}\\[6pt]
 = |x_{0}- y_{0}|^2+\EE\left(\int_{0}^{T}\left|b(t,x(t), u^1(t))-b(t,y(t), u^2(t)) + y_{1}(t)- x_{1}(t)\right|^{2} \dd t\right) \\[6pt]
+\sum_{j=1}^{d}\EE\left(\int_{0}^{T}\left|\sigma^j(t,x(t), u^1(t))-\sigma^j(t,y(t), u^2(t)) + y_2^{j}(t)- x_2^{j}(t)\right|^{2} \dd t\right),
\end{array}$$\normalsize
 which, by the Lipschitz assumption in \eqref{acotamientoderivadas}, is bounded by
 $$ c\left[ \|x- y\|_{\I^{n}}^{2}+ \EE\left( \int_{0}^{T}|x(t)-y(t)|^{2}\dd t\right) + \EE\left( \int_{0}^{T}|u^{1}(t)-u^{2}(t)|^{2}\dd t\right) \right], $$
 for some constant $c>0$. Now,  as in the proof of Lemma \ref{lemmacota1},  by Jensen's and Doob's inequalities we easily get the existence of a constant $c'>0$ such that $$ \EE\left( \int_{0}^{T}|x(t)-y(t)|^{2}\dd t\right) \leq c'  \|x- y\|_{\I^{n}}^{2},$$ from which the Lipschitz property of $g$ easily follows. Now, for $j=1,\hdots,d$  let us set  
$$Db(t,x,u)(z,v)= b_{x}(t,x,u) z +   b_{u}(t,x,u) v, \; \;  D\sigma^j(t,x,u)(z,v)= \sigma_{x}^j(t,x,u) z +   \sigma_{u}^j(t,x,u) v$$ 
 and define 
 $$\begin{array}{l} I_1:= \EE\left(\int_{0}^{T} \left[\frac{b(t, x(t)+ \tau z(t), u(t)+ \tau v(t))- b(t,x(t),u(t))}{\tau}- Db(t,x(t),u(t))(z(t),v(t))\right]^{2}\dd t\right), \\[6pt]
 I_2^j:= \EE\left(\int_{0}^{T} \left[\frac{\sigma^j(t, x(t)+ \tau z(t), u(t)+ \tau v(t))- \sigma^j(t,x(t),u(t))}{\tau}- D\sigma^j(t,x(t),u(t))(z(t),v(t))\right]^{2}\dd t \right). \end{array}$$
By the Lipschitz property of $b$ and $\sigma$ in \eqref{acotamientoderivadas} and the dominated convergence theorem, we get that $I_1$ and $I_2^j$ tend to $0$ as $\tau \downarrow 0$. This implies that
$$(x,u)\in \I^{n}\times (L^{2,2}_{\FF})^{m} \mapsto  \int_{0}^{(\cdot)} b(s,x(s),u(s)) \dd s +  \int_{0}^{(\cdot)} \sigma(s,x(s), u(s)) \dd W(s) \in \I^{n}$$
 is directionally  differentiable with directional derivative
 $$\begin{array}{ll}(z,v)\in \I^{n}\times (L^{2,2}_{\FF})^{m} \mapsto & \int_{0}^{(\cdot)}  Db(t,x(t),u(t))(z(t),v(t))\dd t \\[6pt]
 \; & + \sum_{j=1}^d \int_{0}^{(\cdot)} D\sigma^j(t,x(t),u(t)) (z(t),v(t)) \dd t.
 \end{array}$$
The continuity of the linear application above follows easily from the bounds in the second relation in \eqref{acotamientoderivadas}. Finally, since $(x,u)\in  \I^{n}\times (L^{2,2}_{\FF})^{m} \mapsto   x\in \I^{n}$ is $C^{\infty}$ with derivative  $(z,v)\in  \I^{n}\times (L^{2,2}_{\FF})^{m} \mapsto   z\in \I^{n}$, we obtain \eqref{mmsnbabbasvvwwwaaass}.
 \end{proof} \medskip
 
The previous lemma yields the following result
 \begin{lemma}\label{uniformsurjectivity}  For every  $(x,u)\in \I^{n} \times (L^{2,2}_{\FF})^{m}$ and $\delta \in \I^{n}$, there exists a unique $z\in \I^{n}$ such that $Dg(x,u)(z,0)= \delta$. Moreover, there exists a constant $c>0$, independent of $(x,u,z,\delta)$,   such that $\|z\|_{\I^{n}} \leq  c \|\delta\|_{\I^{n}}$
 \end{lemma}
 \begin{proof}
By Lemma \ref{wrnwnrnwnrnqbqbqbq}, we have that $Dg(x,u)(z,0)=\delta$ is equivalent to the SDE
$$
\begin{array}{rcl} \dd z &=& \left[b_{x}(t,x(t),u(t)) z(t)-\delta_1\right] \dd t + \left[ \sigma_{x}(t,x(t),u(t)) z(t)- \delta_2\right] \dd W(t), \\[4pt]
		z(0)&=& -\delta_0.
\end{array}
$$
The existence and uniquenes of a solution $z$ of this equation is well-known (see e.g. \cite[Chapter 5]{KaraShreve91}). Moreover, using that   $\| b_x\|_{\infty}\leq c_1$ and $ \|\sigma_x\|_{\infty}\leq c_1$,  Lemma \ref{lemmacota1} implies the existence of a constant  $c>0$,  independent of $(x,u,z,\delta)$, such that
$$ \| z\|_{\I^{n}} \leq   c \left[ |\delta_0|^{2}+ \EE\left( \int_{0}^{T} |\delta_1|^{2} \dd t\right)+ \EE\left( \int_{0}^{T} |\delta_2|^{2} \dd t\right)\right].$$
The result follows.
 \end{proof} \smallskip
 
As a consequence of the last two lemmas and Theorem \ref{thm3}, $g$ satisfies \eqref{open} with $C:= \I^n \times \V$  and  $\alpha=c$, where $\V$ is any closed set of $ (L^{2,2}_{\FF})^{m}$. Therefore, the following result holds true.

\begin{corollary}\label{calmness_G_stochastic_control_continuous_time}
For any closed set  $\V\subset  (L^{2,2}_{\FF})^{m}$, we have
$$d((x, u), g^{-1}(y)\cap \left(\I^n\times \V\right)) \leq c\Vert g(x, u) - y\Vert \quad \forall \; x, \; y\in \I^n \; \; \mbox{{\rm and} }  u\in \V.$$
\end{corollary} 
 
 \smallskip
 
Now,  we consider the properties of the cost functional $f$.
 \begin{lemma}\label{F_lipchitz_gateaux} The function $f$ is locally Lipschitz and G\^ateaux differentiable, with  
\begin{equation}\label{formuladerivativeofF}
\begin{array}{rcl}
Df(x,u)(z,v) &=&\EE \left( \int_{0}^{T}\left[\ell_{x}(t,x(t),u(t))z(t)+\ell_{u}(t,x(t),u(t))v(t)\right] \dd t \right)\\[6pt]
\; & \; & + \EE\left( D\Phi(x(T))z(T) \right).
\end{array}
\end{equation}
 \end{lemma}
 \begin{proof}
For $\tau \in [0,1]$, set $x_{\tau}:= x_1+\tau (x_2-x_1)$, $u_{\tau}:= u_1+\tau (u_2-u_1)$, $\delta x= x_2-x_1$ and $\delta u= u_2-u_1$. We have that
 $$
 \begin{array}{ll}
  |f(x_2,u_2)- f(x_1,u_1)| \leq& \EE \left( \int_{0}^{T}  \int_{0}^{1} \left|D\ell(t,x_\tau(t), u_\tau(t))(\delta x(t), \delta u(t))\right|\dd \tau \dd t\right)  \\[6pt]
 \; &+ \EE \left( \int_{0}^{1} \left|D\Phi(x_{\tau}(T))\delta x(T) \right|\dd \tau \right).\end{array}$$
 By the second assumption in \eqref{crecimientoderivadascosto} we can find $c>0$ such that \small
 $$\begin{array}{rcl}
  \left|D\ell(t,x_\tau(t), u_\tau(t))(\delta x(t), \delta u(t))\right| &\leq& c( 1 + |x_\tau(t)|+ | u_\tau(t)|) (|\delta x(t)|+|\delta u(t)|)\\[6pt]
  									\; 	 &\leq& c( 1 +|x_{1}(t)| +  |\delta x(t)|+ |u_1(t)|+  |\delta u(t)|)(|\delta x(t)|+|\delta u(t)|),
\end{array} $$ \normalsize
which, by the Cauchy-Schwarz inequality,  implies that
$$
\begin{array}{l}
\left[\EE \left( \int_{0}^{T}  \int_{0}^{1} \left|D\ell(t,x_\tau(t), u_\tau(t))(\delta x(t), \delta u(t))\right|\dd \tau \dd t\right)\right]^{2} \leq  \\[6pt]
c' \EE \left( \int_{0}^{T}(1+|x_{1}(t)|^2 +  |\delta x(t)|^{2}  +|u_1(t)|^{2}+  |\delta u(t)|^{2}) \dd t\right)(\| \delta x\|_{2,2}^{2}+ \| \delta u \|_{2,2}^{2}).
\end{array}
$$
Analogously, there exists $c''>0$ such that
$$\left[\EE \left( \int_{0}^{1} \left|D\Phi(x_{\tau}(T))\delta x(T) \right|\dd \tau \right)\right]^{2}\leq c''  \EE \left(1+|x_{1}(T)|^2 +  |\delta x(T)|^{2}) \dd t\right)\EE\left( |\delta x(T)|^{2}\right),$$
from which the local Lipschitz property for $f$ follows.  Now, we prove the formula for the directional derivative. Consider the term
\small
\begin{equation}\label{ratioell} \EE\left( \int_{0}^{T} \left[\frac{\ell(t,x(t)+ \tau z(t), u(t)+ \tau v(t))- \ell(t,x(t)+ \tau z(t), u(t)+ \tau v(t))}{\tau}- D\ell(t,x(t),u(t))\right] \dd t \right). \end{equation}\normalsize
Since $\ell$ is G\^ateaux differentiable, the expression inside the integral converges to zero pointwisely. Now, writing the ratio inside the integral in integral form, if $\tau<1$, we have
$$
\int_{0}^{1} D \ell(t, x_{\gamma \tau}(t), u_{\gamma \tau}(t))(z(t), v(t)) \dd \gamma \leq c(1+ |x(t)|+ |z(t)|+|u(t)|+|v(t)|) (|z(t)|+|v(t)|),
$$
where $x_{\gamma \tau}=x + \gamma \tau z$ and $u_{\gamma \tau}= u+\gamma \tau v$. The term $ D\ell(t,x(t),u(t))$ is dominated by $c(1+|x(t)|+ |u(t)|)$ and thus we can pass to the limit to obtain that the term in \eqref{ratioell} tends to $0$ as $\tau\downarrow 0$. Analogously, as $\tau \downarrow 0$,
$$ \EE\left( \frac{\Phi(x(T)+ \tau z(T))- \Phi(x(T))}{\tau} - D\Phi(x(T)) z(T) \right) \to 0.$$
 Formula \eqref{formuladerivativeofF} follows. \end{proof} \medskip

As customary in optimal control theory,   it is convenient to introduce the 
Hamiltonian $H: \Omega \times ]0,T[ \times \R^{n} \times \R^{n} \times \R^{n \times d} \times \R^{m} \to \R$  defined  as
$$ H(\omega, t, x, p, q, u):= \ell(\omega, t,y,u)+ p \cdot b(\omega, t,x,u)+ \sum_{i=1}^{d} q^{i}\cdot \sigma^{i}(\omega, t,x,u).$$

With the help of Theorem \ref{thm2} and Corollary \ref{calmness_G_stochastic_control_continuous_time} we can prove now  a {\it weak version} of the stochastic Pontryagin's minimum principle  (see {\rm\cite{Peng90}} and Remark \ref{comments_result}{\rm(ii)} below).
\begin{theorem}\label{stochasticpmp} Suppose that $(\bar{x},\bar{u})$ is a local solution of problem $\mathcal(SP)$, then there exists $\bar{p} \in \I^{n}$ and $\bar{q} \in (L^{2,2}_{\FF})^{n\times d}$ such that
\begin{equation}\label{weakpmpsystem}
\begin{array}{c}
\bar{p}(\cdot)= \nabla_{x}\Phi(\bar{x}(T)) + \int_{(\cdot)}^{T}\nabla_{x}H(s, \bar{x}(s),\bar{p}(s), \bar{q}(s),\bar{u}(s))\dd s- \int_{(\cdot)}^{T}\bar{q}(s) \dd W(s), \\[6pt]
\mbox{and } \hspace{0.2cm}\EE\left(\int_{0}^{T} \nabla_{u}H (t, \bar{x}(t),\bar{p}(t), \bar{q}(t),\bar{u}(t))\cdot v(t)\dd t \right) \geq 0 \; \; \mbox{for all } \; v \in  T_{\U}(\bar{u}).
\end{array}\end{equation}
If, in addition, $K(\U,\bar{u})$ is convex, then the second relation in \eqref{weakpmpsystem} is valid for all $v\in K(\U,\bar{u})$.
\end{theorem}
\begin{proof}
Lemma \ref{wrnwnrnwnrnqbqbqbq}  and Lemma \ref{F_lipchitz_gateaux} imply that $g$ and $f$ satisfy the assumptions ${\bf (H_g)}$ and ${\bf (H_f)}$, respectively. Since Corollary \ref{calmness_G_stochastic_control_continuous_time} implies that $g$ is calm at $(\bar{x},\bar{u})$, Theorem \ref{thm2} yields  the existence of a Lagrange multiplier $\lambda \in \I^{n}$ such that
$$0 \in Df(\bar{x},\bar{u}) + Dg(\bar{x},\bar{u})^{\ast} \lambda   +  \{0\}\times  N_{\U}(\bar{u}),$$
which can be written as 
\begin{equation}\label{optimality_conditions_separated}
\begin{array}{rcl}
  D_x f(\bar{x},\bar{u}) + D_x g(\bar{x},\bar{u})^{\ast} \lambda&=&0, \\[6pt]
 \left\langle D_u f(\bar{x},\bar{u}) + D_u g(\bar{x},\bar{u})^{\ast} \lambda, v \right\rangle_{(L^{2,2}_{\FF})^m} &\geq &0  \hspace{0.3cm} \forall \; v\in T_{\U}(\bar{u}).
\end{array}
\end{equation}
Setting  $\bar{p}=\lambda_1$ and $\bar{q}=\lambda_2$,  \cite[Theorem 3.12]{backhoffsilva14} implies that the first and second relations in \eqref{weakpmpsystem}  are equivalent to the corresponding relations in   \eqref{optimality_conditions_separated}.  Finally, if $K(\U,\bar{u})$ is convex, by Theorem \ref{thm2} we have  
$$0 \in Df(\bar{x},\bar{u}) + Dg(\bar{x},\bar{u})^{\ast} \lambda   +  \{0\}\times K(\U,\bar{u})^0.$$
Reasoning as before, we have that the second relation in \eqref{weakpmpsystem} is valid for all $v\in K(\U,\bar{u})$. The result follows. 
\end{proof}
\subsection{Comments and extensions}\label{comments_result}
Let us provide some comments on the previous result. 
\begin{itemize}
\item[{\rm(i)}] As pointed out in {\rm\cite{backhoffsilva14}}, it is not clear that in general the function $g$ defined in \eqref{meqmemnnrnnrnr} is $C^1$. Therefore, standard Lagrange multiplier results, in infinite dimensions, are not directly applicable to problem $(SP)$. The results presented in Section \ref{optimality_conditions_abstract_form} and in Section \ref{metric_regularity}  allow us to overcome this difficulty. 
\item[{\rm(ii)}]  It is possible to prove Theorem \ref{stochasticpmp} by following a different strategy that does not involve the Lagrange multiplier theory.  In order to simplify the discussion, we suppose that no constraints are imposed on the controls, i.e. $\U= (L^{2,2}_{\FF})^m$, and refer the interested reader to {\rm\cite{Bonnans_Silva_stochastic}} for the detailed presentation in the general case. Assumption {\bf(A1)} implies that for each $u \in (L^{2,2}_{\FF})^m$, the equation $g(x,u)=0$ admits a unique solution $x[u] \in \I^n$. As a consequence, problem $(SP)$ can be rewritten as the unconstrained optimization problem  
$$
\inf \left\{ J(u):= f(x[u],u) \; ; \; \;  \mbox{{\rm s.t.}} \; \; \;  u \in  (L^{2,2}_{\FF})^m\right\}.  \eqno(SP')
$$ 
If $\bar{u}$ is a local solution of $(SP')$, then it is possible to provide a first order expansion of  $v\in (L^{2,2}_{\FF})^m \to J(\bar{u}+v)$ if $v$ is progressively-measurable and essentially bounded.  By defining $(\bar{p}, \bar{q})$ by the first relation in \eqref{weakpmpsystem} {\rm(}which can be justified by the results in {\rm\cite{Bismut_linear_quadratic})}, the aforementioned expansion of $J$ implies that the second relation in \eqref{weakpmpsystem} holds for every essentially bounded $v$ and so, by a density argument, for every $v\in (L^{2,2}_{\FF})^m$. Even if this approach provides another proof of Theorem \ref{stochasticpmp}, the latter is considerably more technical than the one presented in this article and does not provide the explicit relation between of $\bar{p}$ and $\bar{q}$ and the Lagrange multiplier $\lambda$ associated to the SDE defining the controlled trajectories.

\item[{\rm(iii)}] In the particular case of pointwise control constraints
$$
\U:= \{ u \in (L^{2,2}_{\FF})^m \; ; \; u(\omega,t) \in U \; \; \mbox{a.s}\},
$$
where $U\subseteq \R^m$ is a nonempty closed set, a   result stronger than Theorem \ref{stochasticpmp} has been shown in {\rm\cite{Peng90}}. In this paper,  the author shows that a variation of the Hamiltonian $H$, which involves an additional pair of adjoint processes, is almost surely  pointwisely minimized  at $\bar{u}(\omega,t)$. In this result, no regularity assumptions on the data with respect to  $u$ are imposed. On the other hand, stronger assumptions  with respect to the dependence on the state variable $x$ are assumed {\rm(}which involve strong requirements on the second order derivatives of $\ell$, $\Phi$, $b$ and $\sigma${\rm)}. 
\item[{\rm(iv)}] A straightforward extension of Theorem \ref{stochasticpmp} is the case where the initial point $\hat{x}_0$ is also a decision variable. More precisely, let $\mathcal{X}_0 \subseteq \R^n$ be a closed set and consider the following extension of problem $(SP)$ 
$$\left.\begin{array}{l} \inf_{x,\hat{x}_0,u} \; \EE\left( \int_0^T \ell(\omega,t,x(t),u(t)) \dd t + \Phi(\omega, x(T)) \right)\\[8pt]
\mbox{s.t. } \hspace{0.2cm}\dd x(t) = b(\omega,t,x(t),u(t)) \dd t + \sigma(\omega,t,x(t),u(t)) \dd W(t) \hspace{0.3cm} t\in (0,T), \\[6pt]
\hspace{1cm} x(0)=\hat{x}_0 \in \mathcal{X}_0,\\[6pt]
\hspace{1cm} u \in \mathcal{\U}.
\end{array}\right\} \eqno(SP')$$
Then, this problem can be written in the abstract form 
$$ \inf \; \; f(x,u) \; \;  \mbox{subject to } \;  \tilde{g}(x,u)\in   \I^n \times \mathcal{X}_0, \; \; \;  u\in \U, \eqno(SP')$$
where \small
$$\tilde{g}(x,u):= \left( x(0)+  \int_{0}^{(\cdot)} b(s,x(s),u(s)) \dd s +  \int_{0}^{(\cdot)} \sigma(s,x(s), u(s)) \dd W(s)-x(\cdot),x(0)\right). $$
\normalsize
Suppose that  $(\bar{x},\bar{u}) \in \I^n \times \mathcal{U}$ is a local solution to $(SP')$ and assume that {\bf(A1)}-{\bf(A2)} hold true.  Using the surjectivity property of the derivative of the first coordinate of $\tilde{g}$ (as in Lemma \ref{uniformsurjectivity}), it is easy to check that \eqref{open1} in ${ \bf(H_{cq}')}$ is satisfied at $(\bar{x},\bar{u})$ (with $C= \I^n \times \mathcal{U}$ and $D= \I^n \times \mathcal{X}_0$). Thus, by Theorem \ref{casDconvexe}, Theorem \ref{thm2_b}, and reasoning as in the proof of Theorem \ref{stochasticpmp}, we obtain the existence of $\bar{p} \in \I^{n}$ and $\bar{q} \in (L^{2,2}_{\FF})^{n\times d}$ such that
\begin{equation}\label{weakpmpsystem_variable_initial_condition}
\begin{array}{l}
\bar{p}(\cdot)= \nabla_{x}\Phi(\bar{x}(T)) + \int_{(\cdot)}^{T}\nabla_{x}H(s, \bar{x}(s),\bar{p}(s), \bar{q}(s),\bar{u}(s))\dd s- \int_{(\cdot)}^{T}\bar{q}(s) \dd W(s), \\[6pt]
-\bar{p}(0) \in N_{\mathcal{X}_0}(\bar{x}(0)),\\[6pt]
\mbox{and } \hspace{0.2cm}\EE\left(\int_{0}^{T} \nabla_{u}H (t, \bar{x}(t),\bar{p}(t), \bar{q}(t),\bar{u}(t))\cdot v(t)\dd t \right) \geq 0 \; \; \mbox{for all } \; v \in  T_{\U}(\bar{u}).
\end{array}\end{equation}

\item[{\rm(v)}] Another easy extension is the case where finitely many final constraints on the state, in expectation form, are added to problem $(SP)$. In this case, a qualification condition has to be imposed on the local solution $(\bar{x},\bar{u})$ in order to ensure that ${ \bf(H_{cq}')}$ holds. We refer the reader to {\rm\cite{backhoffsilva14}} for a more detailed discussion on this matter. The case of final pointwise constrains having the form $x(\omega,T) \in \mathcal{X}_{T}$, for some closed set $\mathcal{X}_T \subseteq  \R^n$, and with probability one,  remains as an interesting open problem. 
\end{itemize}

\section{Application to a class of stochastic control problems in discrete time}\label{section_discrete_time}
Let $(\Omega, \F, \PP)$ be a probability space and, as in the previous section, denote by $\EE$ the expectation under $\PP$. Let $w_1, \hdots, w_{N}$ be $N$ independent  $\R^d$-valued random variables defined in $(\Omega, \F, \PP)$ such that for all $k=1, \hdots, N$ the coordinates of $w_k=(w_k^1, \hdots, w_k^d)$ are independent and satisfy
$$
\EE(w_k^i)= 0, \hspace{0.3cm} \EE(|w_k^i|^2)= 1.
$$
Define  $w_0:= 0$ and for $k=0,\hdots, N$ set  $\F_{k}:=\sigma\left(w_0, \hdots, w_{k}\right)$, the sigma-algebra generated by $w_0, \hdots, w_{k}$, and  
$$
L^{2}_{\F_k}:= \left\{ y\in L^{2}(\Omega) \; ;  \; y \;   \mbox{is $\F_k$ measurable} \right\}.
$$
Let $\U \subseteq  \Pi_{k=0}^{N-1} (L^{2}_{\F_k})^m$ be a non-empty closed set. In this section we consider the following discrete-time stochastic optimal control problem (see \cite{LinZhang})
$$
\left.
\begin{array}{l}\inf  \;   \EE\left(\sum_{k=0}^{N-1}\ell(k,x_k,u_k)  + \Phi(x_N)\right)\\[8pt]
\mbox{s.t. } \hspace{0.2cm}  x_{k+1}= b(k,x_k, u_k) + \sigma(k,x_k, u_k)w_{k+1} \hspace{0.5cm} k=0, \hdots, N-1  \\[8pt]
\hspace{0.85cm} x_0= \hat{x}_0\in \R^n \\[8pt] 
\hspace{0.85cm} x \in \Pi_{k=0}^{N} (L^{2}_{\F_k})^n, \; \;  u \in \U, 
\end{array}\right\} \eqno(SP_d)
$$
where, denoting $ [0: N-1]:=\{0,\hdots, N-1\}$, $\ell: [0: N-1] \times \R^{n} \times \R^{m} \to \R$, $\Phi: \R^{n} \to \R$, $b: [0: N-1] \times \R^{n} \times \R^{m} \to \R^{n}$ and $\sigma: [0: N-1] \times \R^{n} \times \R^{m} \to \R^{n\times d}$   are Borel measurable functions.  Denoting $\sigma^{j}$ ($j=1,\hdots, d$) the $j$th column of $\sigma$, for $\psi=b$, $\sigma^j$ we suppose that $\psi$ is $\mathcal{C}^1$ with respect to $(x,u)$ and the existence of $c_{1}>0$ such that for all $k\in [0:N-1]$
\begin{equation}\label{crecimientoderivadasdinamica_caso_discreto}\left\{\begin{array}{c} |\psi(k,x,u)|\leq  c_{1} \left(  1+ |x|+|u|\right), \\[4pt]
																			|\psi_{x}(k,x,u)|+ |\psi_{u}(k,x,u)| \leq c_{1}.  				
\end{array}\right.
\end{equation}
Similarly, in the remainder of this section we will assume that there exists $c_2>0$ such that for all $k\in [0:N-1]$
\begin{equation}\label{crecimientoderivadascosto_caso_discreto}\left\{\begin{array}{c}|\ell (k,x,u)|\leq  c_{2} \left( 1+ |x|+|u|\right)^{2}, \\[4pt]
 |\ell_{x}(k,x,u)|+ |\ell_{u}(k,x,u)|\leq  c_{2} \left( 1+ |x|+|u|\right), \\[4pt]
 \ |\Phi (x)|\leq  c_{2} \left( 1+ |x|\right)^{2}, \; |\Phi_{x}(x)|\leq  c_{2} \left( 1+ |x|\right).
 \end{array}\right.\end{equation}
As in Section \ref{section_stochastic_control_continuous_time} we introduce now a Hilbert space for the state $x$ which is suitable for the application of the results in Sections \ref{optimality_conditions_abstract_form} and \ref{metric_regularity}.  Set $X_0=\R^n$ and  given $k\in [1:N]$ define 
$$
X_k:= \left\{ y^{0}_{k-1} +  \sum_{i=1}^{d} y^{i}_{k-1}w_k^{i} \; ; \;  y_{k-1}^i \in   \left(L^{2}_{\F_{k-1}}\right)^{n} \hspace{0.2cm} \; \forall \; i=0,\hdots,d \right\}.
$$
Endowed with the scalar product  
$$
\langle x, x' \rangle_{X_k} := \EE\left(  \sum_{i=0}^{d}  y_{k-1}^{i} \cdot  z_{k-1}^{i}   \right) \hspace{0.4cm} \forall \;  x=y_{k-1}^{0}+ \sum_{i=1}^dy_{k-1}^{i} w_{k}^i  , \; \; \;  x'=z_{k-1}^{0}+ \sum_{i=1}^dz_{k-1}^{i}  w_{k}^i, 
$$
the following elementary result shows that $X_k$ is a Hilbert space. 
\begin{lemma}For every $(y_{k-1}^0,y_{k-1}^1,\hdots, y_{k-1}^d) \in (L^{2}_{\F_{k-1}})^n \times (L^{2}_{\F_{k-1}})^{n \times d}$ we have
\begin{equation}\label{isometry}
\EE\left( \left|y_{k-1}^{0}+ \sum_{i=1}^dy_{k-1}^{i} w_{k}^i\right|^2 \right)= \sum_{i=0}^{d} \EE\left(|y_{k-1}^{i}|^2\right).
\end{equation}
As a consequence, for every $k\in [1:N]$ the linear operator $I: (L^{2}_{\F_{k-1}})^n \times (L^{2}_{\F_{k-1}})^{n \times d} \to X_k$ defined as 
$$
I(y_{k-1}^0,y_{k-1}^1,\hdots, y_{k-1}^d) :=  y^{0}_{k-1} +  \sum_{i=1}^{d} y^{i}_{k-1}w_k^{i},
$$
is a bijection.   
\end{lemma}
\begin{proof} Relation \eqref{isometry} follows directly from the relations
$$\begin{array}{l}
\EE\left( y_{k-1}^{0} \cdot y_{k-1}^{i} w_{k}^i\right)=\EE\left( y_{k-1}^{0} \cdot y_{k-1}^{i} \EE\left(w_{k}^i|\F_{k-1}\right)\right)=0 \hspace{0.3cm} \forall \; i \in [1:d], \\[6pt]
\EE\left( y_{k-1}^{i} \cdot y_{k-1}^{j} w_{k}^jw_{k}^i \right)=\EE\left( y_{k-1}^{i} \cdot y_{k-1}^{j} \EE\left(w_{k}^jw_{k}^i|\F_{k-1}\right) \right)= \left\{\begin{array}{ll} \EE\left( |y_{k-1}^{i}|^2\right) & \mbox{if } i= j, \\
									0 & \mbox{otherwise}. \end{array}\right.
\end{array}
$$

By definition of $X_k$ we only need to show that $I$ is injective. But this is clear because if 
$$
I(y_{k-1}^0,y_{k-1}^1,\hdots, y_{k-1}^d)=0,
$$
then \eqref{isometry} implies that $\EE\left(|y_{k-1}^i|^2\right)=0$  and so $y_{k-1}^i=0$ a.e. for all $i\in [0:d]$.
%
%
\end{proof}\smallskip\\

 Define $g: \Pi_{k=0}^{N} X_k \times  \Pi_{k=0}^{N-1} \left(L^{2}_{\F_k}\right)^m \to \Pi_{k=0}^{N} X_k$ as 
$$\begin{array}{rcl}
g_0(x,u)&:=& \hat{x}_0 - x_0, \\[6pt]
 g_{k+1}(x,u) &:=& b(k, x_k, u_k) + \sigma(k,x_k,u_k) w_{k+1} - x_{k+1} \hspace{0.5cm} \forall \;  k=0,\hdots, N-1, 
\end{array}$$
and $f: \Pi_{k=0}^{N} X_k \times  \Pi_{k=0}^{N-1}   \left(L^{2}_{\F_k}\right)^m \to \R$ as
$$
f(x,u):=\EE\left(\sum_{k=0}^{N-1}\ell(k,x_k,u_k)  + \Phi(x_N)\right).
$$
Under these notations, problem $(SP_d)$ can be rephrased as 
$$ \inf \; \; f(x,u) \; \;  \mbox{subject to } \;  g(x,u)=0, \; \; \;  u\in \U. \eqno(SP_d)$$
As in the previous section, we prove now that if we set $X:=\Pi_{k=0}^{N} X_k \times  \Pi_{k=0}^{N-1} \left(L^{2}_{\F_k}\right)^m$, then under our assumptions the mappings $f$ and $b$ satisfy the assumptions   in Section  \ref{optimality_conditions_abstract_form}. 
\begin{lemma}\label{gateaux_differentiability_dynamics_discrete_time} The following assertions hold true:\smallskip\\
{\rm(i)} The mapping $g$ is  Lipschitz and G\^ateaux differentiable. For $(x,u)$, $(z,v)\in X$ the directional derivative of $g$ at $(x,u)$ in the direction $(z,v)$ is given by $Dg(x,u)(z,v)= \left(Dg_0(x,u)(z,v), \hdots, Dg_N(x,u)(z,v)\right)$, where
\begin{equation}\label{gateaux_derivative_discrete_case}\begin{array}{rcl}
Dg_0(x,u)(z,v)&=& -z_0,\\[6pt]
 Dg_{k+1}(x,u)(z,v)&=&  b_{(x,u)}(k,x_k,u_k)(z_k,v_k)+ \\[6pt]
 \; & \; & \sum_{i=1}^{d} \sigma_{(x,u)}^i(k,x_k,u_k)(z_k,v_k) w_{k+1}^{i} - z_{k+1},
\end{array}
\end{equation}
for all $k=0,\hdots,N-1$. \smallskip\\
{\rm(ii)} The mapping $f$ is locally Lipschitz and G\^ateaux differentiable, with 
\begin{equation}\label{gateaux_differentiability_cost_discrete_time_expression}
Df(x,u)(z,v)= \EE\left( \sum_{k=0}^{N-1} \ell_{(x,u)}(k,x_k,u_k)(z_k,v_k) + D\Phi(x_N)z_N \right),
\end{equation}
for all $(x,u)$, $(z,v)\in X$.
\end{lemma}
\begin{proof} We only prove assertion {\rm(i)} since the proof of {\rm(ii)} is analogous. By the second relation in assumption \eqref{crecimientoderivadasdinamica_caso_discreto}, there exists $c>0$ such that for all $k=0,\hdots, N-1$,
$$\begin{array}{l}
\|g_{k+1}(x^1,u^1) -  g_{k+1}(x^2,u^2)\|_{X_{k+1}}^{2}\\[6pt]
=\EE\left(|b(k, x^1_k,u^1_k) - b(k, x^2_k,u^2_k)|^ 2 + \sum_{i=1}^{d}|\sigma^i(k, x^1_k,u^1_k) - \sigma^i(k, x^2_k,u^2_k)|^ 2 \right)\\[6pt]
\leq c \EE\left( |x_k^1 -  x_k^2|^2 + |u_k^1 -  u_k^2|^2\right)= c\left( \|x_k^1 -  x_k^2\|_{X_k}^2 + \|u_k^1- u_k^2\|_{L_{\F_k}^{2}}^2 \right),
\end{array}
$$
where the last equality follows from \eqref{isometry}. The Lipschitz continuity  of $g$ easily follows. Now, for $\psi=b, \sigma^{i}$ ($i=1,\hdots, d$) we have
$$\begin{array}{l}
\EE \left( \frac{\psi(k,x_k +\tau z_k, u_k+ \tau v_k)-\psi(k,x_k, u_k)}{\tau} - \psi_{x}(k,x_k,u_k)z_k- \psi_{u}(k,x_k,u_k)v_k \right)^2 \to 0,
\end{array}
$$
by the Lipschitz continuity of $\psi(k,\cdot,\cdot)$ and the Lebesgue's dominated convergence theorem. The continuity of the linear mapping $(z,v) \to Dg(x,u)(z,v)$ follows easily from \eqref{gateaux_derivative_discrete_case}, assumption \eqref{crecimientoderivadasdinamica_caso_discreto} and the isometry \eqref{isometry}.
\end{proof} \smallskip

As a corollary of the first assertion in the previous lemma, we obtain the following result. 
\begin{lemma}\label{metric_regularity_discrete_time_case} For every $(x,u) \in X$ and $\delta \in \Pi_{k=0}^{N} X_{k}$ there exists a unique $z \in \Pi_{k=0}^{N} X_{k}$ such that $Dg(x,u)(z,0)= \delta$. Moreover, there exists $c>0$, independent of $(x,u,z,\delta)$, such that 
\begin{equation}\label{uniform_surjectivity_discrete_time}
\sum_{k=0}^{N} \| z_k\|_{X_k} \leq c \sum_{k=0}^{N} \| \delta_k\|_{X_k}.
\end{equation}
In particular, for every closed set $\V\subseteq \Pi_{k=0}^{N-1}  \left(L^{2}_{\F_k}\right)^m$ we have that
\begin{equation}\label{calmness_property_stochastic_case_discrete_time}
d\left((x,u), g^{-1}(y) \cap \left( \Pi_{k=0}^{N} X_{k} \cap \V \right) \right) \leq c \hspace{0.3cm} \forall \; (x,u) \in X, \; \; y \in \Pi_{k=0}^{N} X_{k}.
\end{equation}
\end{lemma}
\begin{proof} The unique $z\in \Pi_{k=0}^{N} X_k$ such that $Dg(x,u)(z,0)=\delta$ is given recursively by 
$$\begin{array}{rcl}
z_0&=& -\delta_0 \\[6pt]
z_{k+1}&=& b_x(k,x_k,u_k)z_k + \sum_{i=1}^{d}\sigma_x^i(k,x_k,u_k)z_k w_{k+1}^i -\delta_{k+1} \hspace{0.6cm} \forall \;  k=0,\hdots, N-1.
\end{array}
$$
Noting that \small
$$\begin{array}{rcl}
\|z_{k+1}\|_{X_{k+1}}^2=\EE\left(|z_{k+1}|^2\right) &\leq& (d+2)\left[c_1^2 \EE\left(|z_k|^2\right)+c_1^2 \sum_{i=1}^{d}\EE\left(|z_{k}|^2(w_{k+1}^i)^2\right)+|\delta_{k+1}|^2 \right], \\[6pt]
\; & \leq  & (d+2)^2c_1^2\EE\left(|z_k|^2\right)+ (d+2)\EE(|\delta_{k+1}|^2),\\[6pt]
\; & \leq  & \bar{c} \left[ \EE\left(|z_k|^2\right)+ \EE(|\delta_{k+1}|^2)\right]\\[6pt]
\; & \leq & (N+1)\bar{c}^{N+1} \sum_{k=0}^{N}\EE\left(|\delta_k|^2\right)\\[6pt]
\; &=&(N+1)\bar{c}^{N+1} \sum_{k=0}^{N}\|\delta_k\|_{X_k}^2 ,
\end{array}
$$\normalsize
where $\bar{c}:=(d+2)^2(c_1^2 + 1)>1$ and the last equality is a consequence of \eqref{isometry}. This proves \eqref{uniform_surjectivity_discrete_time}. Relation \eqref{calmness_property_stochastic_case_discrete_time} follows directly from \eqref{uniform_surjectivity_discrete_time} and Theorem \ref{thm3}.
\end{proof} \smallskip
%

Let us define the Hamiltonian $H:[0: N-1]\times \R^n \times \R^n \times \R^{n\times d} \times \R^m \to \R$ by
$$
H(k,x,p,q,u):=\ell(k,x,u) +p  \cdot b(k,x,u) +\sum_{i=1}^{d}q^{i} \cdot \sigma^{i}(k,x,u)
$$
We have now all the elements to establish the optimality system for problem $(SP_d)$. 
\begin{theorem}\label{optimality_system_discrete_case} Suppose that $(\bar{x},\bar{u})$ is a local solution to $(SP_d)$. Then, there exist  $p \in \Pi_{k=0}^{N-1} (L^{2}_{\F_{k}})^n$, $q \in \Pi_{k=0}^{N-1} (L^{2}_{\F_{k}})^{n\times d}$ such that  
\begin{equation}\label{identidad_adjunto_discreto}
\begin{array}{rcl}
p_{k-1}&=& \EE\left( \nabla_{x} H(k,\bar{x}_{k},p_{k},q_{k},\bar{u}_{k}) | \F_{k-1}\right)\hspace{0.2cm} \forall \; k\in [1:N-1] \\[6pt]
q_{k-1}^{i}&=& \EE\left( \nabla_{x} H(k,\bar{x}_{k},p_{k},q_{k},\bar{u}_{k}) w_{k}^i | \F_{k-1}\right) \hspace{0.2cm} \forall \; k\in [1:N-1], \; i \in [1:d]  \\[4pt]
p_{N-1} &=&\EE \left( \nabla \Phi(\bar{x}_N) | \F_{N-1} \right)\\[4pt]
q_{N-1}^i &=&\EE \left( \nabla \Phi(\bar{x}_N) w_{N}^{i}| \F_{N-1} \right) \hspace{0.2cm} \forall \; i \in [1:d],
\end{array}
\end{equation}  
and 
\begin{equation}\label{optimality_condition_u_discrete}
\EE\left( \sum_{k=1}^{N-1} \nabla_{u}H(k,\bar{x}_k,\bar{p}_{k},\bar{q}_{k},\bar{u}_k) \cdot  v_k  \right) \geq 0  \hspace{0.3cm} \forall \; v \in T_{\U}(\bar{u}).
\end{equation}
If in addition    $K(\U,\bar{u})$  is convex, then \eqref{optimality_condition_u_discrete} holds for all $v\in K(\U,\bar{u})$.
\end{theorem}
\begin{proof} By Lemma \ref{gateaux_differentiability_dynamics_discrete_time}  and  Theorem \ref{thm2} there exists $\lambda \in \Pi_{k=0}^{N} X_k$ such that 
$$
(0,0) \in Df(\bar{x},\bar{u}) + Dg(\bar{x},\bar{u})^{\ast}\lambda + \{0\}\times N_{\U}(\bar{u}),
$$
from which we deduce that for all $z=(z_0,\hdots, z_{N}) \in  \Pi_{k=0}^{N} X_k$
\begin{equation}\label{lagrange_abstract_discrete_problem}
\begin{array}{l}
D_{x_k} f(\bar{x},\bar{u})z_k + \sum_{j=0}^{N} \langle \lambda_j, D_{x_k} g_j(\bar{x},\bar{u})z_k \rangle_{X_j}= 0 \hspace{0.3cm} \; \forall \; k=0,\hdots, N,  \\[6pt]
D_u f(\bar{x},\bar{u})v + \sum_{k=1}^{N} \langle \lambda_k, D_u g_{k}(\bar{x},\bar{u})v \rangle_{X_k}\geq 0  \hspace{0.3cm} \forall v \in T_\U(\bar{u}). 
\end{array}
\end{equation}
Lemma \ref{gateaux_differentiability_dynamics_discrete_time}  and the first equation in \eqref{lagrange_abstract_discrete_problem} imply that for all $k=1,\hdots,N-1$
\begin{equation}\label{optimality_equation_with_lambda_k}
\begin{array}{c}
\EE\left(\ell_{x}(k,\bar{x}_k,\bar{u}_k)z_k\right)+ \left\langle \lambda_{k+1}, b_{x}(k,\bar{x}_k,\bar{u}_k)z_k+\sum_{i=1}^{d}w_{k+1}^i\sigma_{x}^{i}(k,\bar{x}_k,\bar{u}_k)z_k \right\rangle_{X_{k+1}}\\[8pt]
\hspace{2cm} =\langle \lambda_k,z_k\rangle_{X_k},\\[8pt]
\EE\left(\Phi_x(\bar{x}_N)z_N\right) =\langle\lambda_{N},z_N \rangle_{X_N}.
\end{array}
\end{equation}
Setting 
$$z_{k}=y_{k-1}^0 + \sum_{i=1}^{d}y_{k-1}^i w_{k}^i \in X_k, \; \; \lambda_{k}=p_{k-1} + \sum_{i=1}^{d}q_{k-1}^i w_{k}^i \in X_k, $$
relation \eqref{optimality_equation_with_lambda_k} yields \small
\begin{equation}\label{intermediate_equation_1_discrete}
\begin{array}{l}
\EE\left(  \nabla_{x} H(k+1,\bar{x}_{k+1},p_{k+1},q_{k+1},\bar{u}_{k+1}) \cdot z_k  \right)= \EE\left(p_{k-1}\cdot y_{k-1}^0+\sum_{i=1}^{d}q_{k-1}^{i}\cdot y_{k-1}^i  \right),\\[6pt]
\hspace{3.67cm}\EE\left(\nabla \Phi(\bar{x}_N)\cdot z_N\right) = \EE\left(  p_{N-1}\cdot y_{N-1}^0+\sum_{i=1}^{d}q_{N-1}^{i}\cdot y_{N-1}^i \right).
\end{array}
\end{equation} \normalsize
Taking $y_{k-1}^{i}= 0$ for all $i \in [1:d]$ the first equation in \eqref{intermediate_equation_1_discrete} gives 
$$
\EE\left(  \nabla_{x} H(k+1,\bar{x}_{k+1},p_{k+1},q_{k+1},\bar{u}_{k+1}) \cdot y_{k-1}^0  \right)= \EE\left(p_{k-1}\cdot y_{k-1}^0\right),
$$
and so, since $y_{k-1}^0\in L^2_{\F_{k-1}}$ is arbitrary, by definition of conditional expectation w.r.t. $\F_{k-1}$, the first equality in 
\eqref{identidad_adjunto_discreto} follows. Similarly, fixing $\bar{i}\in [1:d]$ and letting $y_{k-1}^{i}=0$ for all $i \in [0:d] \setminus \{\bar{j}\}$, we obtain the second relation \eqref{identidad_adjunto_discreto} for $i=\bar{i}$. The last two relations in  \eqref{identidad_adjunto_discreto} follow by an analogous argument.

Finally, since for all $k=0,\hdots, N-1$,
$$
\langle \lambda_{k+1}, D_u g_{k+1}(\bar{x},\bar{u})v \rangle_{X_{k+1}}=\EE\left(p_{k}\cdot b_{u}(k,\bar{x}_k, \bar{u}_k)v_k +\sum_{i=1}^{d}q_{k}^{i} \cdot \sigma_{u}^{i}(k,\bar{x}_k, \bar{u}_k)v_k \right),
$$
relation \eqref{optimality_condition_u_discrete} follows directly from the second relation in \eqref{lagrange_abstract_discrete_problem} and Lemma \ref{gateaux_differentiability_dynamics_discrete_time}{\rm(ii)}.  If $K(\U,\bar{u})$  is convex then Theorem \ref{thm2} ensures that the second relation in \eqref{lagrange_abstract_discrete_problem} holds for all $v\in K(\U,\bar{u})$, from which the last assertion of the theorem easily follows. 
\end{proof}
\begin{remark}{\rm(i)} The optimality system \eqref{identidad_adjunto_discreto}-\eqref{optimality_condition_u_discrete} has been first shown in {\rm\cite{LinZhang}} under more restrictive assumptions on $\ell$, $\Phi$, $f$, $\sigma$ {\rm(}see {\rm\cite[Equation (9)]{LinZhang}}{\rm)}  and the control constraint set $\U$ {\rm(}see {\rm\cite[Section 3]{LinZhang})}. The results in Sections \ref{optimality_conditions_abstract_form} and \ref{metric_regularity} allow us to  prove a more general result in a quite direct manner. \smallskip\\
{\rm(ii)} Similarly to the continuous case {\rm(}see Section \ref{section_stochastic_control_continuous_time}{\rm)}, it is easy to extend the results in this section to the case where the initial state $\hat{x}_0$ is a decision variable subject to the constraint  $\hat{x}_0 \in \mathcal{X}_0$, where $\mathcal{X}_0$ is a closed subset of $\R^n$. In this case, the optimality system is as in Theorem \ref{optimality_system_discrete_case} with the additional constraint  on the adjoint state (called transversality condition) $-p_0 \in  N_{\mathcal{X}_0}(\hat{x}_0)$. 
\end{remark}

\end{document}